\documentclass[hidelinks,onefignum,onetabnum]{siamart220329}

\usepackage{amsfonts}
\usepackage{graphicx}
\usepackage{epstopdf}
\usepackage{algorithm}
\usepackage{cite}
\usepackage{algpseudocode}
\usepackage{enumitem}
\usepackage{verbatim}
\usepackage[textsize=scriptsize]{todonotes}
\usepackage{ifoddpage}
\usepackage{dsfont}
\usepackage{subcaption}
\usepackage{setspace}

\ifpdf
  \DeclareGraphicsExtensions{.eps,.pdf,.png,.jpg}
\else
  \DeclareGraphicsExtensions{.eps}
\fi

\newsiamremark{remark}{Remark}
\newsiamremark{hypothesis}{Hypothesis}
\crefname{hypothesis}{Hypothesis}{Hypotheses}
\newsiamthm{claim}{Claim}

\headers{BQO systems and balanced truncation}{H. Fa\ss bender, S. Gugercin and T. Peters}

\title{Bilinear quadratic output systems and balanced truncation\thanks{Submitted to the editors DATE.
\funding{Part of this work was completed during a research stay of the third author at Virginia Tech supported by a fellowship of the German Academic Exchange Service (DAAD). The work of Gugercin was supported in part by the National Science Foundation (NSF), United States under Grant No.
CMMI-2130695}}}

\author{Heike Fa\ss bender \thanks{Institute for Numerical Analysis, TU Braunschweig
  (\email{h.fassbender@tu-braunschweig.de}).} 
\and Serkan Gugercin\thanks{Department of Mathematics and Division of Computational Modeling and Data Analytics, Academy of Data Science, Virginia Tech
  (\email{gugercin@vt.edu}).}\and Till Peters\thanks{Institute for Numerical Analysis, TU Braunschweig 
  (\email{till.peters@tu-braunschweig.de}), corresponding author.}}

\usepackage{amsopn}

\newcommand{\R}{\mathbb{R}}
\newcommand{\C}{\mathbb{C}}
\newcommand{\norm}[1]{\left\lVert#1\right\rVert}

\ifpdf
\hypersetup{
  pdftitle={Bilinear quadratic output systems and related model reduction techniques},
  pdfauthor={H. Fa\ss bender, S. Gugercin and T. Peters}
}
\fi

\begin{document}

\maketitle

\begin{abstract}
Dynamical systems with quadratic outputs have recently attracted significant attention. In this paper, we consider bilinear dynamical systems, a special class of weakly nonlinear systems, with a quadratic output. We develop various primal-dual formulations for these systems and define the corresponding system Gramians. Conditions for the existence and uniqueness of these Gramians are established, and the generalized Lyapunov equations they satisfy are derived. Using these Gramians and their truncated versions, 
which are computationally more efficient, we construct a balanced truncation framework for bilinear systems with quadratic outputs. The proposed approach is demonstrated through two numerical examples.    
\end{abstract}

\begin{keywords}
    Model order reduction, balanced truncation, bilinear systems, linear quadratic output systems,
    Gramians, generalized Lyapunov equations
\end{keywords}

\begin{MSCcodes}
93A15, 93B05, 93B07, 93C10, 93C15 
\end{MSCcodes}

\section{Introduction}
We investigate the so-called bilinear quadratic output (BQO) dynamical systems described in the state-space form as
\begin{subequations}\label{eq:BQO}
    \begin{align}
    \dot{x}(t) &= Ax(t) + \sum_{k=1}^m N_k x(t) u_k(t)+ Bu(t), \qquad x(0)=0, \label{eq:BQO_a1}\\
    y(t)&= Cx(t)+\begin{bmatrix}
        x(t)^TM_1x(t) \\ \vdots \\ x(t)^TM_px(t)
    \end{bmatrix}, \label{eq:BQO_a2}
\end{align}
\end{subequations}
where $A\in\R^{n\times n}, B\in\R^{n \times m}, C\in \R^{p\times n}, N_k\in\R^{n\times n}$ for $k=1,\dots, m$, $M_j\in \R^{n\times n}$ for $j=1,\dots, p$, and $t\in T:=[0,\infty)$.  In~\eqref{eq:BQO}, $x(t)\in\R^n$ describes the state, $u(t)\in \R^m$ the inputs, and $y(t)\in \R^p$ the outputs of the system. It can be assumed without loss of generality that $M_j$ is symmetric since $x(t)^TM_jx(t)=x(t)^TM_j^Tx(t)=\frac{1}{2}x(t)^T(M_j+M_j^T)x(t)$.

The BQO systems represent a generalization of bilinear systems (see, e.g., \cite{al-baiyat_new_1993,zhang_2002}) as well as of linear quadratic output (LQO) systems (see, e.g., \cite{benner_lqo_2022}). More specifically,
in \cref{eq:BQO_a2}, setting $M_1=\dots=M_p=0$ leads to the output $y(t) = Cx(t)$, recovering the bilinear dynamical systems. On the other hand,
setting $N_1=\dots=N_m=0$ in \cref{eq:BQO_a1} leads to the state equation  $\dot{x}(t) = Ax(t)+Bu(t)$, 
recovering the linear quadratic output (LQO) systems.

Bilinear state equations naturally appear in various applications or a result of Carleman bilinearization of nonlinear dynamics, see, e.g., \cite{rugh_nonlinear_1981,mohler,hart13}. In recent years, dynamical systems with quadratic outputs as in \cref{eq:BQO_a2} have gained significant attention as they appear in the modeling of the variance of a quantity of interest of a stochastic model or in measuring the energy of state; see, e.g., \cite{PulA19,Pul23,YueM13,ReiW24,VanVNLM12}
and the references therein.
Moreover, in the context of port-Hamiltonian dynamical systems,  bilinear port-Hamiltonian dynamical systems
naturally appear with a state-space form
\begin{subequations}\label{eq:BPH} 
\begin{align}
    \dot{x}(t) &= (J-R)Qx(t) + \left(B+Nx(t)\right)u(t),\\
    y(t) &= \left(B+Nx(t)\right)^TQx(t),
\end{align}
\end{subequations}
where $J,R,Q,N\in\R^{n \times n}, B\in\R^{n\times 1}$ with $J=-J^T, R=R^T\geq 0, Q=Q^T\geq 0$, see~\cite{MehU23,Van06}. (For simplicity, here we only considered a single-input single-output bilinear port-Hamiltonian system.) Note that \cref{eq:BPH} is a special case of \cref{eq:BQO}. 

In most cases, such as those where the dynamical system arises from the spatial discretization of an underlying PDE, the dimension of the resulting BQO system is significantly large, and the task of model order reduction arises.
Thus we seek a BQO system of similar structure but  with a significantly smaller dimension which should accurately describe the same input-output behavior. More precisely, we look for the reduced BQO system 
of order $r \ll n$ described as
\begin{subequations}\label{eq:BQOr}
    \begin{align}
    \dot{\hat{x}}(t) &= \hat{A}\hat{x}(t) + \sum_{k=1}^m \hat{N}_k \hat{x}(t) u_k(t)+ \hat{B}u(t), \qquad \hat{x}(0)=0, \label{eq:BQOr_a1}\\
    \hat{y}(t)&=\hat{C}\hat{x}(t)+\begin{bmatrix}
        \hat{x}(t)^T\hat{M}_1\hat{x}(t) \\ \vdots \\ \hat{x}(t)^T\hat{M}_p\hat{x}(t)
    \end{bmatrix}, \label{eq:BQOr_a2}
\end{align}
\end{subequations}
where $\hat{A}\in\R^{r\times r}, \hat{B}\in\R^{r \times m}, \hat{C}\in \R^{p\times r}$, $\hat{N}_k\in\R^{r\times r}$ for $k=1,\dots, m$, $\hat{M}_j\in \R^{r\times r}$ for $j=1,\dots, p$, $t\in T$. In \cref{eq:BQOr}, $\hat{x}(t)\in\R^r$ describes the reduced state and $\hat{y}(t)\in \R^p$ the reduced output of the system. The goal is to construct \cref{eq:BQOr} such that
the reduced output $\hat{y}(t)$ provides a high-fidelity approximation to the original output $y(t)$.

In this manuscript, we employ \emph{balanced truncation} (BT), a model order reduction method that is well established for linear systems (originally in \cite{moore_bt,mullis_roberts}, see also \cite{antoulas_book_2005,BenB17,BreS21,GugA04}), or \cite[Section 7 and 8]{benner_modellreduktion_2024}). BT transforms the original system into a state-space representation where the reachability and observability Gramians are equal and diagonal. The diagonal elements are called the Hankel singular values. The states corresponding to the smallest Hankel singular values are then removed because they have the least influence on the system's behavior. In general, this results in a reduced model that provides a good approximation of the original system with a lower dimension. Theoretical foundation of balancing and balanced truncation for nonlinear systems has been introduced in \cite{scherpen_balancing_1993} which requires solving Hamilton-Jacobi partial differential equations.
This task is computationally challenging for large-scale systems and the reduced system, in general, will not inherit the original structure due to the state-dependent model reduction bases. For example, in our setting, this means that the reduced system will not retain the structure~\cref{eq:BQOr}. As an alternative, one seeks algebraic Gramians, which are generalizations of the linear case and solve generalized Lyapunov equations.  One of the main contributions of this work will be the development of proper and relevant algebraic Gramians for BQO systems.

BQO systems and their model order reduction have been briefly discussed in the master's thesis \cite[pp. 25–42]{padhi_2024}. In particular, \cite[Section 3.1]{padhi_2024} highlights a transformation of BQO systems into bilinear systems with linear output via lifting transformations. While this approach enables the application of existing bilinear system theory and reduction techniques, it comes at the significant cost of increasing the system matrix size from $n \times n$ to $(n+n^2) \times (n+n^2)$ - an increase in the dimension that is clearly impractical for large-scale systems. Therefore, instead, \cite{padhi_2024} avoids this lifting technique and aims to develop reduction tools directly on the original BQO structure. To achieve this goal in the BT setting,
one needs the Gramians.  In \cite{padhi_2024}, the reachability Gramian of \eqref{eq:BQO} is taken to be the same as in the bilinear case, 
that is, \eqref{eq:BQO} with $M_1=\dots=M_p=0$. Similarly, the construction of the observability Gramian of \eqref{eq:BQO} is motivated by the observability Gramian for the bilinear case. 
The analysis in \cite{padhi_2024} is restricted to single-input single-output (SISO) configurations ($m = p = 1$).

In contrast, our work addresses the more general case of multi-input multi-output (MIMO) BQO systems and focuses on developing the reachability and observability Gramians following the approach for quadratic bilinear systems from \cite{BenG24}. In quadratic bilinear systems, \eqref{eq:BQO_a1} has an additional quadratic summand $H(x(t) \otimes x(t)),$ while in \eqref{eq:BQO_a2} all matrices $M_j, j = 1, \dots , p$ are zero. Therefore, with $H = 0$, the derivation of the reachability Gramian given in \cite{BenG24} can be adopted here. 

In \cite[Section 2.2]{BenG24}, it is suggested to define the observability Gramian for quadratic bilinear systems as the reachability Gramian of \emph{the dual system}, analogous to what is known for linear and bilinear systems. We will follow this approach in our derivations. As we will show in \Cref{subsec:3-2},
there are multiple ways to write the BQO system \eqref{eq:BQO} as a primal system.
This allows us to introduce multiple variants of the observability Gramian, each satisfying a different (generalized) Lyapunov equation. 
Interestingly, although the observability Gramian in \cite{padhi_2024} is derived using a different approach than the one considered here, 
it satisfies the same Lyapunov equation as one of the observability Gramians derived here. As we will show, since this equation has a unique symmetric positive semidefinite solution under certain conditions, the two Gramians must be identical. 

The structure of the paper is as follows: Section \ref{sec:2} introduces preliminaries. In Section \ref{sec:3}, we derive the reachability and observability Gramians for BQO systems. Section \ref{subsec:3-1} summarizes the derivation of the reachability Gramian for quadratic bilinear systems from \cite{BenG24} adapted to the case of BQO systems. In Section \ref{subsec:3-2}, the primal and the dual systems for the BQO system \eqref{eq:BQO} are discussed, showing that multiple  
 ways exist to write the primal system, which leads to different dual systems. The corresponding observability Gramians are derived in Sections \ref{subsubsec:3-2}, \ref{subsubsec:3-3}, {and \ref{subsubsec:3-4}}. 
Section \ref{sec:5} provides the theoretical foundations, including existence and uniqueness results for the Gramians and their Lyapunov equations. In Section \ref{sec:6}, we discuss Gramian truncation for code acceleration and analyze the resulting Lyapunov equations. Section \ref{sec:7} applies balanced truncation to BQO systems, comparing the interpretability of different observability Gramians and introducing a tailored algorithm. Finally, Section \ref{sec:8} presents numerical results using various Gramian combinations.

\section{Preliminaries}\label{sec:2}

We will use the standard Householder notation labeling matrices with upper-case Roman letters ($A, B,$ etc.), vectors with lower-case Roman letters ($a, b,$ etc.), and scalars with lower-case Greek letters ($\alpha, \beta$, etc.). When exposing columns or rows of a matrix, the columns are usually labeled with the corresponding Roman lower case letter, indexed with an integer $i,j,k, m$ or $p$.

Throughout this paper, we will use the spectral norm $\norm{\cdot}=\norm{\cdot}_2$. We write $\sigma(T) \subset \mathbb{C}$  for the spectrum of a linear operator $T$ and $\rho(T) = \max\{|\lambda| \mid \lambda \in \sigma(T)\}$ for the spectral radius.
We use $X \geq 0$ to denote a symmetric positive semidefinite matrix $X \in \mathbb{R}^{n \times n},$ while $X > 0$ denotes a symmetric positive definite matrix $X$.
  Furthermore, we say that $A$ is stable if all eigenvalues of $A$ lie in the left half-plane.

Moreover, we will use the standard Kronecker product notation as defined in, e.g., \cite[Section 4.2,4.3]{HorJ94} or \cite[pp.27-28]{GolubVanLoan4th}). In particular, we will use the fact that
\[(A\otimes B)(C \otimes D) = AB \otimes CD\]
holds for matrices of appropriate size, and
$
\norm{I \otimes X} = \norm{X}$.
In addition, we will use the following result:
\begin{lemma} \label{lem:lem1}
    Let $R\in\R^{n\times n}$ and 
    $\mathcal{G}:=\begin{bmatrix}
        G_1 \dots G_q 
    \end{bmatrix}, \mathcal{H}:=\begin{bmatrix}
        H_1^T  \dots  H_q^T 
    \end{bmatrix}^T$,
    with $G_k,H_k\in\R^{n\times n}$ for all $k=1,\dots, q$. Then, it holds
    \begin{align*}
        \mathcal{G}(I_q\otimes R)\mathcal{H}=\sum_{k=1}^qG_kRH_k.
    \end{align*}
\end{lemma}
The structure from \Cref{lem:lem1} will often appear and be exploited later; so given 
$N_k$ and $M_j$ appearing in the BQO system dynamics \cref{eq:BQO}, we define the matrices
\begin{subequations}\label{eq:mathcal}
    \begin{align}
    \mathcal{N}:=&\begin{bmatrix}
        N_1 & \dots & N_m
    \end{bmatrix}\in \R^{n\times nm}, \label{eq:mathcalN}\\
    \mathcal{N}_T:=&\begin{bmatrix}
        N_1^T & \dots & N_m^T
    \end{bmatrix}\in \R^{n\times nm},\label{eq:mathcalNt}\\
    \mathcal{M}:=&\begin{bmatrix}
        M_1 & \dots & M_p
    \end{bmatrix}\in \R^{n\times np}.\label{eq:mathcalM}
\end{align}

\end{subequations}

\section{Reachability and Observability Gramians} \label{sec:3}
We seek to apply model reduction to BQO systems using balanced truncation, that is, by balancing
the system and then truncating the unimportant states afterward. This will require properly defining the Gramians. Toward this goal, we will define and use generalized Gramians. 
Because BQO systems \eqref{eq:BQO} are a generalization of bilinear systems and of linear quadratic output systems, the Gramians derived here 
will cover the Gramians of those systems as a special case.

\subsection{Reachability Gramian}\label{subsec:3-1}
We begin by deriving the reachability Gramian $P$ of the BQO system \cref{eq:BQO}. This has already been discussed in a more general setting in \cite[Section 2.1]{BenG24} in the context of quadratic bilinear systems (see, also, \cite[Section II]{al-baiyat_new_1993}). For quadratic bilinear systems, the state equation \eqref{eq:BQO_a1} has an additional term in the form of $H(x(t)\otimes x(t))$. In the following, we give the equations and definitions from \cite{BenG24} relevant for our further discussion for the case $H= 0$, thus recovering the bilinear dynamics.  
The concept of Volterra series expansion plays the fundamental role in deriving the algebraic reachability Gramian for nonlinear systems. To simplify the notation appearing in these expansions (see below), we define $u^{(t_1,\dots,t_l)}_{k}(t):=u_k(t-t_1-\dots -t_l)$ and $x^{(t_1,\dots,t_l)}(t):= x(t-t_1-\dots -t_l)$,  where $u_k$ denotes the $k$th element of the input vector $u$.

As shown in \cite{BenG24}, the solution of the state equation \cref{eq:BQO_a1} is given by 
\begin{align}
    x(t) = \int_0^te^{At_1}Bu^{(t_1)}(t)dt_1+\sum_{k=1}^m\int_0^t e^{At_1} N_kx^{(t_1)}(t)u^{(t_1)}_{k}(t)dt_1, \label{eq:xt1}
\end{align} 
which can be rewritten as $x(t)=\sum_{i=1}^\infty x_i(t)$ with

  \begin{equation}\label{eq:xi_expr} \small
    x_i(t)    
     =\int_0^t\int_0^{t-t_i}\dots \int_0^{t-t_i-\dots- t_{2}} \bar{P}_i(t_1,\dots,t_i)\left( u^{(t_i)}(t) \otimes \dots \otimes u^{(t_i,\dots,t_1)}(t)\right) dt_1\dots dt_i,
\end{equation}
where, using $\mathcal{N}$ from \cref{eq:mathcal}, the matrices 
$P_i$ are 
\begin{subequations}
\label{eq:Pibar}
\begin{align}
    \bar{P}_1(t_1) &= e^{At_1}B \in \mathbb{R}^{n\times m}, \\
    \bar{P}_2(t_1,t_2) &= e^{At_2}\mathcal{N} (I_m\otimes \bar{P}_1(t_1)) \in \mathbb{R}^{n\times m^2},\\
    &\vdots \notag \\
    \bar{P}_{i}(t_1,\dots, t_i) &= e^{At_i}\mathcal{N} (I_m\otimes \bar{P}_{i-1}(t_1,\dots,t_{i-1}))\in \mathbb{R}^{n\times m^i}.
\end{align}
\end{subequations}
 
Now, we are ready to define the reachability Gramian $P$ for \eqref{eq:BQO}.
\begin{definition}(Reachability Gramian $P$)
The reachability Gramian $P$ for the BQO system \eqref{eq:BQO} is defined as
\begin{align} \label{eq:sumpi}
P=\sum_{i=1}^{\infty}P_i \in\R^{n\times n}
\end{align}
with
\begin{equation}\label{eq:Pi}
    P_i=\int_0^\infty\dots\int_0^\infty \bar{P}_{i}(t_1,\dots, t_i) \bar{P}_{i}^T(t_1,\dots, t_i) dt_1\dots dt_i
\end{equation}
where $\bar{P}_{i}$ as defined in \cref{eq:Pibar}, 
assuming that all integrals exist and the infinite series converge.
\end{definition}
 To conclude, we adapt \cite[Thm. 1]{BenG24} (see also \cite[Thm. 1]{al-baiyat_new_1993}) to our setting to show that if $P$ \cref{eq:sumpi} exists, then it satisfies a generalized Lyapunov equation.
\begin{theorem}[\cite{BenG24}]
    \label{thm:thm1}
    Let a BQO system \eqref{eq:BQO} with a stable matrix $A$ be given. If the reachability Gramian $P$ of the BQO system defined as in (\ref{eq:sumpi}) exists, then it satisfies the generalized Lyapunov equation  
    \begin{equation}
        AP+PA^T+\sum_{k=1}^mN_kPN_k^T +BB^T =0 \label{eq:thm1a}.
    \end{equation}
    Moreover, $P_i$ defined in \cref{eq:Pi} satisfies
    \begin{subequations}\label{eq:P_rec}
        \begin{align}
            AP_1+P_1A^T+BB^T=0,\\
            AP_i+P_iA^T+\sum_{k=1}^mN_kP_{i-1}N_k^T=0 \quad \text{for } i\geq 2.
        \end{align}
        \end{subequations}
        Also, define $\hat{P}_\ell=\sum_{i=1}^\ell P_i$. Then, 
        \begin{equation}\label{eq:P_fixit}
 A\hat{P}_\ell+\hat{P}_\ell A^T+\sum_{k=1}^mN_k\hat{P}_{\ell-1}N_k^T+BB^T=0 \quad \text{for } \ell\geq 2.
 \end{equation}
\end{theorem}

The equations \cref{eq:thm1a}-\cref{eq:P_fixit} will prove crucial in proving the existence of a unique positive (semi)definite solution of \eqref{eq:thm1a} and in computing $P$ via a fixed-point iteration.

\subsection{Observability Gramian}\label{subsec:3-2}
Next we derive the observability Gramian of a BQO system as the reachability Gramian of the dual system using an approach similar to the one used in \cite{BenG24} for bilinear quadratic dynamical systems. For the \emph{primal system} 
\begin{subequations}\label{eq:primal}
\begin{align}
    \dot{x}(t) &= \mathcal{A}(x,u,t)x(t)+ \mathcal{B}(x,u,t)u(t), \quad x(0)=0, \\
    y(t) &= \mathcal{C}(x,u,t)x(t)+\mathcal{D}(x,u,t)u(t),
\end{align}
\end{subequations}
the associated \emph{dual system} is given by
\begin{align}
    \dot{x}(t) &= \mathcal{A}(x,u,t) x(t) + \mathcal{B}(x,u,t) u(t), \quad x(0)=0,\nonumber \\
    \dot{z}(t) &= -\mathcal{A}^T(x,u,t)z(t)-\mathcal{C}^T(x,u,t)\hat{u}(t), \quad z(\infty)=0, \label{eq:dual}\\
    \hat{y}(t)&=\mathcal{B}^T(x,u,t)z(t)+\mathcal{D}^T(x,u,t)\hat{u}(t), \nonumber
   \end{align}
where $z(t)\in\R^n$ describes the dual state, $\hat{u}(t)\in\R^p$ the dual input and $\hat{y}(t)\in\R^m$ the dual output; see, e.g., \cite{fujimoto_hamiltonian_2000}.

We have 
(at least) two ways to write the BQO system \eqref{eq:BQO} as a primal system of the form \eqref{eq:primal}. 
The somewhat straightforward choice is to consider the bilinear part of the state equation \eqref{eq:BQO_a1} as part of $\mathcal{A}(x,u,t)$, in other words, we choose 
\begin{equation}\label{eq-dual-standard}
\begin{aligned}
    \mathcal{A}(x,u,t) &= A+\sum_{k=1}^mN_ku_k(t), &
    \mathcal{B}(x,u,t)= B,\\
    \mathcal{C}(x,u,t)&= C+\begin{bmatrix}
        x(t)^TM_1 \\\vdots \\ x(t)^TM_p
    \end{bmatrix}, &
    \mathcal{D}(x,u,t)=0.
\end{aligned}
\end{equation}
A distinctly different possibility is to consider the bilinear part as part of $\mathcal{B}(x,u,t)$:
\begin{equation}\label{eq-dual-different}
\begin{aligned}
    \mathcal{B}(x,u,t) & = B+\begin{bmatrix}N_1x(t) & \dots & N_mx(t)\end{bmatrix}, & 
    \mathcal{A}(x,u,t) = A,\\
    \mathcal{C}(x,u,t)&= C+\begin{bmatrix}
        x(t)^TM_1 \\\vdots \\ x(t)^TM_p
    \end{bmatrix}, &
    \mathcal{D}(x,u,t)=0.
\end{aligned}
\end{equation}
One could also use a mixture of \eqref{eq-dual-standard} and \eqref{eq-dual-different} and define
\begin{equation}\label{eq-dual-mix}
\begin{aligned}
    \mathcal{A}(x,u,t) &= A+\sum_{k=1}^m \phi_kN_ku_k(t), \\
    \mathcal{B}(x,u,t) &= B +\begin{bmatrix}(1-\phi_1) N_1x(t) & \dots & (1-\phi_m) N_mx(t)\end{bmatrix},\\
    \mathcal{C}(x,u,t)&= C+\begin{bmatrix}
        x(t)^TM_1 \\\vdots \\ x(t)^TM_p
    \end{bmatrix}, \text{~~~and~~~}
    \mathcal{D}(x,u,t)=0
\end{aligned}
\end{equation}
for given constants $\phi_k, k = 1, \dots, m$ (or even matrices $\phi_k \in \mathbb{R}^{n \times n}$). While choosing $\phi_k=1$ yields the primal system \cref{eq-dual-standard}, choosing $\phi_k=0$ yields \cref{eq-dual-different}.

These variants of the primal system lead to different dual systems, and thus to different observability Gramians. This is explained in more detail in \Cref{subsubsec:3-2,subsubsec:3-3}  for the choices \eqref{eq-dual-standard} and \eqref{eq-dual-different}. In addition, the resulting generalized Lyapunov equations which are satisfied by the corresponding observability Gramians are derived. We show that although the observability Gramian in \cite{padhi_2024} was derived in a completely different way than the one we do  \Cref{subsubsec:3-2}, it satisfies the same Lyapunov equation as the Gramian derived \Cref{subsubsec:3-2}. Since that Lyapunov equation has a unique symmetric positive semidefinite solution under some assumptions, the two Gramians must therefore be identical. These considerations are explained in more detail in  \Cref{subsubsec:3-2} and \Cref{sec:5}.  The proofs of the existence of the different Gramians as well as the proof of the existence and uniqueness of the solution of the Lyapunov equations related to these Gramians are covered in \Cref{sec:5}.

\subsubsection{Observability Gramian using choice \cref{eq-dual-standard}}\label{subsubsec:3-2} 
First we interpret the BQO system \eqref{eq:BQO} as the primal system \eqref{eq:primal} with the choice \eqref{eq-dual-standard}. Then using \eqref{eq:dual}, we obtain the dual system
\begin{align}
    \dot{x}(t) &= Ax(t) +\sum_{k=1}^mN_kx(t)u_k(t) +Bu(t), \quad x(0)=0,\nonumber \\
    \dot{z}(t) &= -A^Tz(t)-\sum_{k=1}^m N_k^Tz(t)u_k(t)-C^T\hat{u}(t)-\sum_{j=1}^pM_jx(t)\hat{u}_{j}(t), \quad z(\infty)=0, \label{eq:dual2}\\
    \hat{y}(t)&=B^Tz(t). \nonumber
  \end{align}
The dual state $z(t)$ as a solution to \cref{eq:dual2} is given as 
\begin{align*}    
    z(t) &= \int^t_\infty e^{-A^T(t-t_1)}\Bigg[ C^T\hat{u}(t_1) + \sum_{k=1}^m N_k^Tz(t_1)u_k(t_1)
    + \sum_{j=1}^pM_jx(t_1)\hat{u}_{j}(t_1) \Bigg] dt_1 \\
        &= \int^t_\infty e^{-A^T(t-t_1)}\Bigg[ C^T\hat{u}(t_1) + \mathcal{N}_T\left(u(t_1)\otimes z(t_1)\right)
    +\mathcal{M}(\hat{u}(t_1)\otimes x(t_1)) \Bigg] dt_1,
\end{align*}
where $\mathcal{N}_T$ and $\mathcal{M}$ are as in \cref{eq:mathcal}.
By a change of variables $t_1\leftrightarrow t+t_1$, this last integral can be expressed as  

\begin{equation}  \label{eq:z_solu}
\begin{aligned}
z(t) &=
 \int_\infty^0 e^{A^Tt_1}\Bigg[ C^T\hat{u}(t+t_1) + \mathcal{N}_T\left(u(t+t_1)\otimes z(t+t_1)\right) \\
      &\qquad \qquad \qquad\qquad +\mathcal{M}\left(\hat{u}(t+t_1)\otimes x(t+t_1)\right)\Bigg] dt_1.
 \end{aligned}
 \end{equation}
In the discussion for the reachability Gramian $P$ in \Cref{subsec:3-1}, we used a Volterra series expansion of the state as $x(t)=\sum_{i=1}^\infty x_i(t)$ with $x_i(t)$ as in \cref{eq:xi_expr} where the index $i$ corresponds to the number of appearing exponential terms in the expressions. There are no longer dependencies of states of different order. We now use a similar ansatz for the dual state as $z(t)=\sum_{i=1}^\infty z_i(t)$, where the term $z_i(t)$ should correspond to a term with only exponential terms of order $i$. If we now substitute $x(t)=\sum_{i=1}^\infty x_i(t)$ and $z(t)=\sum_{i=1}^\infty z_i(t)$ into \cref{eq:z_solu}, we obtain
\begin{align*}
    z(t)=&
 \int_\infty^0 e^{A^Tt_1}\Bigg[ C^T\hat{u}(t+t_1) + \mathcal{N}_T\left(u(t+t_1)\otimes \left(\sum_{i=1}^\infty z_i(t+t_1)\right)\right) \\
      &\qquad \qquad \qquad\qquad +\mathcal{M}\left(\hat{u}(t+t_1)\otimes \left(\sum_{i=1}^\infty x_i(t+t_1)\right)\right)\Bigg] dt_1 \\
      =&
 \int_\infty^0 e^{A^Tt_1} C^T\hat{u}(t+t_1) dt_1 \\
 &+  \sum_{i=1}^\infty \int_\infty^0 e^{A^Tt_1}\Biggl[ \mathcal{N}_T\left(u(t+t_1)\otimes  z_i(t+t_1)\right) \\
      &\qquad \qquad \qquad\qquad +\mathcal{M}\left(\hat{u}(t+t_1)\otimes  x_i(t+t_1)\right)\Bigg] dt_1.
\end{align*}
By comparing the orders of exponential terms, we can now define the terms $z_i(t)$ as
\begin{align*}
    z_1(t)  &= \int^0_\infty e^{A^Tt_1} C^T\hat{u}(t+t_1) dt_1, \quad \mbox{for}~i=1, \\ 
    z_i(t)  &= \int^0_\infty e^{A^Tt_i}\begin{bmatrix}\mathcal{N}_T \Bigl(I_m\otimes z_{i-1}(t+t_{i})\Bigr)&\mathcal{M}\Bigl(I_p\otimes x_{i-1}(t+t_i)\Bigr)\end{bmatrix}\begin{bmatrix}
        u(t+t_i) \\\hat{u}(t+t_i)
    \end{bmatrix}dt_i,
\end{align*}
for $i \geq 2$. This motivates to define the observability mapping as
\begin{align*}
    \bar{Q}^S = \begin{bmatrix}
        \bar{Q}^S_1&\bar{Q}^S_2 & \dots&
    \end{bmatrix},
\end{align*}
where the $\bar{Q}^S_i$'s are defined as
\begin{equation} \label{eq:Qisbar}
    \begin{aligned}
    \bar{Q}^S_1(t_1) &= e^{A^Tt_1}C^T \in \mathbb{R}^{n\times p},\\
    \bar{Q}^S_2(t_1,t_2) &= e^{A^Tt_2}\begin{bmatrix}\mathcal{N}_T \Bigl(I_m\otimes \bar{Q}^S_1(t_1)\Bigr)  
    &\mathcal{M}\Bigl(I_p\otimes \bar{P}_1(t_1)\Bigr) \end{bmatrix} \in \mathbb{R}^{n\times 2pm}, \\
    &\vdots \\
    \bar{Q}^S_{i}(t_1,\dots, t_i) &= e^{A^Tt_i}\Bigl[\mathcal{N}_T \Bigl(I_m\otimes \bar{Q}^S_{i-1}(t_1,\dots,t_{i-1})\Bigr), \\
    &\qquad \qquad \qquad \qquad     \mathcal{M} \Bigl(I_p\otimes \bar{P}_{i-1}(t_1,\dots,t_{i-1})\Bigr) \Bigr]\in \mathbb{R}^{n\times 2pm^{i-1}}.
\end{aligned}
\end{equation}
With the observability mapping properly established, 
we define the (standard) observability Gramian $Q^S$ for the BQO system.
\begin{definition}[(Standard) Observability Gramian $Q^S$]
The (standard) observability Gramian $Q^S$ for the BQO system is defined as
\begin{align}
     \label{eq:sumqi}
Q^S=\sum_{i=1}^{\infty}Q^S_i \in\R^{n\times n},
\end{align}
with
\begin{equation}\label{eq:QSi}   Q^S_i=\int_0^\infty\dots\int_0^\infty \bar{Q}^S_{i}(t_1,\dots, t_i) \left(\bar{Q}^S_{i}(t_1,\dots, t_i)\right)^T dt_1\dots dt_i,
\end{equation}
where $\bar{Q}^S_{i}$ as defined in \cref{eq:Qisbar},  assuming that all integrals exist and the infinite series converge.
\end{definition}
Next, we show that assuming $P$ and $Q^S$ exist, $Q^S$ satisfies a generalized Lyapunov equation, which depends on the reachability Gramian $P.$
\begin{theorem}
    \label{thm:thm2}
     Let a BQO system \eqref{eq:BQO} with a stable matrix $A$ be given. Assume that
    the reachability Gramian $P$ and the observability Gramian $Q^S$ of the BQO system \eqref{eq:BQO} defined as in (\ref{eq:sumpi}) and (\ref{eq:sumqi}) exist. 
    Then the Gramian $Q^S$  satisfies the generalized $P$-dependent Lyapunov equation
    \begin{equation}
        A^TQ^S+Q^SA+\sum_{k=1}^mN_k^TQ^SN_k+\sum_{j=1}^pM_jPM_j  +C^TC=0.\label{eq:thm1b}
    \end{equation}
    \end{theorem}
    \begin{proof}     
     To prove \cref{eq:thm1b}, we study the expansion \eqref{eq:sumqi} of $Q^S$. It is well known that (see, e.g., \cite[\S 3.8]{ZhoDG96}, \cite{Son98}, \cite[Thm. 7.25]{benner_modellreduktion_2024})
        \begin{align*}
            Q_1^S&=\int_0^\infty \bar{Q}^S_1(t_1)\bar{Q}^S_1(t_1)^Tdt_1 =\int_0^\infty e^{A^Tt_1}C^TCe^{At_1} dt_1
        \end{align*}
        solves the Lyapunov equation.
        \begin{equation}\label{eq:QS_reca}
        A^TQ^S_1+Q^S_1A+C^TC=0. 
        \end{equation}
       For $i\geq 2$, we have 
        \begin{align*}
            Q^S_i &= \int_0^\infty \dots \int_0^\infty \bar{Q}^S_i\left(\bar{Q}^S_i\right)^Tdt_1\dots dt_i,\\
            &= \int_0^\infty \dots \int_0^\infty e^{A^Tt_i} \begin{bmatrix} \mathcal{N}_T\Bigl(I_m\otimes\bar{Q}^S_{i-1}\Bigr) & \mathcal{M}\Bigl(I_p\otimes \bar{P}_{i-1}\Bigr)\end{bmatrix} \\
            &\qquad \cdot \begin{bmatrix}\mathcal{N}_T\Bigl(I_m\otimes\bar{Q}^S_{i-1}\Bigr) & \mathcal{M}\Bigl(I_p\otimes \bar{P}_{i-1}\Bigr)\end{bmatrix}^T e^{At_i} dt_1\dots dt_i.
        \end{align*}
Multiplying out yields 
            \begin{align*}
            Q^S_i&= \int_0^\infty \dots \int_0^\infty e^{A^Tt_i} \left[ \mathcal{N}_T\Bigl(I_m\otimes\bar{Q}^S_{i-1}\Bigr)\left(I_m\otimes\bar{Q}^S_{i-1}\right)^T\mathcal{N}_T^T \right.\\
            &\qquad \qquad + \left. \mathcal{M}\Bigl(I_p\otimes \bar{P}_{i-1}\Bigr)\Bigl(I_p\otimes \bar{P}_{i-1}\Bigr)^T\mathcal{M}^T \right] dt_1\dots dt_i\\            
            &= \int_0^\infty e^{A^Tt_i} \left[ \mathcal{N}_T\Bigl(I_m\otimes Q^S_{i-1}\Bigr)\mathcal{N}_T^T + \mathcal{M}\Bigl(I_p\otimes P_{i{-1}}\Bigr)\mathcal{M}^T\right]e^{At_i}dt_i.
            \end{align*}
            Using \Cref{lem:lem1},  we see that $Q^S_i$ solves
            \begin{align}\label{eq_Lyap_Qi}
                A^TQ^S_i+Q^S_iA+\sum_{k=1}^m N_k^T Q^S_{i-1} N_k+ \sum_{j=1}^pM_jP_{i-1}M_j =0.
            \end{align}
         Summing over \eqref{eq:QS_reca} and \eqref{eq_Lyap_Qi} for $i = 2, 3, \dots, \ell$, we have, for $\hat Q^S_\ell:=\sum_{i=1}^\ell Q^S_i$, that
\begin{equation}\label{eq:QS_hat}
    A^T \hat Q^S_\ell + \hat Q^S_\ell A + \sum_{k=1}^m N_k^T\hat Q^S_{\ell-1}N_k + \sum_{j=1}^p M_j \hat P_{\ell-1}M_j +C^TC = 0
\end{equation}
            holds with $\hat{P}_i$ as in \cref{eq:P_fixit}.
As $Q^S:=\sum_{i=1}^\infty Q^S_i = \lim_{\ell \rightarrow \infty} \hat Q^S_\ell$, this proves \cref{eq:thm1b}.      
    \end{proof}
    
The equations \eqref{eq:QS_reca} and \eqref{eq:QS_hat} will be helpful in proving the existence of a positive (semi)definite solution of \eqref{eq:thm1b} and in computing $Q^S$ via a fixed-point iteration.

The Lyapunov equation \eqref{eq:thm1b} already appears in \cite{padhi_2024}. There, however, a completely different approach was used to derive the observability Gramian, which we denote by $Q^{P}$ in order to distinguish it from $Q^S$.      
      The construction of $Q^P$ is based on the observability Gramian $Q^B$ for the bilinear case (that is, \eqref{eq:BQO} with $M_1 = \cdots = M_p=0$), which can be constructed as follows (see, e.g., \cite{zhang_2002}). Define 
         \begin{align*}
            \bar{Q}_1^B(t_1) &= e^{A^Tt_1}C^T \in \mathbb{R}^{n\times p},\\
            \bar{Q}_{i}^B(t_1,\dots, t_i) &= e^{A^Tt_i}\mathcal{N}_T(I_m\otimes \bar{Q}^B_{i-1}(t_1,\dots,t_{i-1})) \in \mathbb{R}^{n\times pm^{i-1}}.
         \end{align*}
        Then, the observability Gramian $Q^B$ of the bilinear system (\cref{eq:BQO} with $M_1=\dots = M_p=0$) is defined as
        \begin{align} \label{eq:sumqi_B}
        Q^B=\sum_{i=1}^{\infty}Q^B_i \in\R^{n\times n} 
        \end{align}
        with
        \begin{align*}
    Q_i^B=\int_0^\infty\dots\int_0^\infty \bar{Q}^B_{i}(t_1,\dots, t_i) \left(\bar{Q}^B_{i}(t_1,\dots, t_i)\right)^T dt_1\dots dt_i
        \end{align*}
        assuming that all integrals exist and the infinite series converge. If $Q^B$ exists, it is a solution of the generalized Lyapunov equation \eqref{eq:thm1b} with $M_1=\cdots=M_p=0$
        \[
        A^TQ^B+Q^BA+\sum_{k=1}^mN_k^TQ^BN_k+C^TC=0.
        \]
        With this perspective, the observability Gramian $Q^P$ is defined as follows.
        \begin{definition}{(Observability Gramian $Q^P$ \cite[Definition 3.3.2]{padhi_2024})}\label{ourdef}
            Let $Q^B$ be the bilinear observability Gramian from \cref{eq:sumqi_B}. For $j=1,2,\dots$ let
            \begin{align*}
                \bar{Q}_{1,j}(s_1,t_1,\dots, t_j)&= e^{A^Ts_1}\mathcal{M}(I_p\otimes \bar{P}_j(t_1,\dots,t_j)), \\
                \bar{Q}_{i,j}(s_1,\dots,s_{i},t_1,\dots t_j)&= e^{A^Ts_i}\mathcal{N}_T(I_m\otimes \bar{Q}_{i-1,j}(s_1,\dots,s_{i-1},t_1,\dots t_j)) \quad \text{for }i>1.
                \end{align*}
                Define
                \begin{align} \label{eq:sumqi_P}
                    Q^P=Q^B+\sum_{i,j=1}^\infty Q_{i,j}
                \end{align}
                with
                \begin{align*}       
                Q_{i,j} &= \int_0^\infty \dots \int_0^\infty \bar{Q}_{i,j}(s_1,\dots,s_{i},t_1,\dots t_j) \\
                &\qquad \qquad \cdot \left( \bar{Q}_{i,j}(s_1,\dots,s_{i},t_1,\dots t_j) \right)^T dt_1\dots dt_jds_1\dots ds_i.
            \end{align*}
        \end{definition}
        In \cite[Prop. 3.3.1, Prop. 3.3.2 and Cor. 3.3.3]{padhi_2024}, it is shown that 
        $Q^P$ satisfies  
        the generalized Lyapunov equation \eqref{eq:thm1b}. Please note that in \cite{padhi_2024} only the SISO case ($m=p=1$) is considered. In this manuscript, Definition \ref{ourdef} as well as \eqref{eq:sumqi_B} are a generalization to the MIMO case. 
           
        If $Q^S$ and $Q^P$ both exist, and there is a unique symmetric, positive semidefinite solution to equation \eqref{eq:thm1b}, then $Q^S$ and $Q^P$ are equivalent, which we will prove in \Cref{sec:5}.
        Comparing the construction of $Q^S$ and $Q^P$, we also observe that if the associated Lyapunov equations have unique solutions, then
        \begin{align*}
            Q_1^S&=Q_1^B,\\
            Q_i^S&=Q_i^B + \sum_{l=1}^{i-1}Q_{l,i-l}.
        \end{align*}

\begin{remark}\label{rem:gen_standard}  
    It has already been mentioned in passing that the results presented here include the case for bilinear systems as a special case with $M_1=\dots=M_p=0$, where the Lyapunov equations are given as
    \begin{subequations}
    \label{eq:bil_lyap_eq}
    \begin{align}    
        AP+PA^T+\sum_{k=1}^mN_kPN_k^T+BB^T&=0, \label{eq:bil_lyap_eq1}\\
        A^TQ+QA+\sum_{k=1}^mN_k^TQN_k+C^TC&=0. \label{eq:bil_lyap_eq2}
    \end{align}
    \end{subequations}
    LQO systems  with $N_1=\dots = N_m=0$ also appear as a special case. In that case the corresponding Lyapunov equations are given as
    \begin{subequations}\label{eq:lqo_lyap_eq}
    \begin{align}
        AP+PA^T+BB^T&=0, \label{eq:lqo_lyap_eq1}\\
        A^TQ+QA+\sum_{j=1}^pM_jPM_j +C^TC&=0. \label{eq:lqo_lyap_eq2}
    \end{align}
    \end{subequations}   
    \end{remark}

\subsubsection{Observability Gramian using choice \cref{eq-dual-different}}
\label{subsubsec:3-3}
In this section, we will develop an alternative observability Gramian using the primal system \cref{eq:primal} with the choice \cref{eq-dual-different}.
Then using \eqref{eq:dual}, this leads to the dual system
\begin{align}
    \dot{x}(t) &= Ax(t) +\sum_{k=1}^mN_kx(t)u_k(t) +Bu(t), \quad x(0)=0,\nonumber \\
    \dot{z}(t) &= -A^Tz(t)-C^T\hat{u}(t)-\sum_{j=1}^pM_jx(t)\hat{u}_{j}(t), \quad z(\infty)=0, \label{eq:dual3}\\
    \hat{y}(t)&=B^Tz(t)+\begin{bmatrix}
        x(t)^TN_1^T \\ \vdots \\ x(t)^TN_m^T
    \end{bmatrix}z(t). \nonumber
   \end{align}
The dual state $z(t)$ as a solution to \cref{eq:dual3} is given as 
\[
    z(t) = \int^t_\infty e^{-A^T(t-t_1)}\Bigg[ C^T\hat{u}(t_1) 
    + \sum_{j=1}^pM_jx(t_1)\hat{u}_{j}(t_1) \Bigg] dt_1. 
\]
A discussion analogous to the one in Section \ref{subsubsec:3-2} leads to the definition of the observability mapping
\begin{align*}
    \bar{Q}^A = \left[
        \bar{Q}^A_1,~\bar{Q}^A_2,~\dots
    \right],
\end{align*}
with the $\bar{Q}_i^A$'s defined as
\begin{equation} \label{eq:QiAsbar}
\begin{aligned}
    \bar{Q}_1^A(t_1) &= e^{A^Tt_1}C^T \in \mathbb{R}^{n\times p},\\
    \bar{Q}_2^A(t_1,t_2) &= e^{A^Tt_2}\mathcal{M} (I_p\otimes \bar{P}_1(t_1)) \in \mathbb{R}^{n\times pm}, \\
    &\vdots \\
    \bar{Q}_{i}^A(t_1,\dots, t_i) &= e^{A^Tt_i}\mathcal{M} (I_p\otimes \bar{P}_{i-1}(t_1,\dots,t_{i-1})) \in \mathbb{R}^{n\times pm^{i-1}}.
\end{aligned}
\end{equation}
With the observability mapping, 
we define the observability Gramian $Q^A$ for the BQO system \cref{eq:BQO}.
\begin{definition}{(Alternative)  Observability Gramian $Q^A$)}
The (alternative) ob\-serv\-ability Gramian $Q^A$ for the BQO system is defined as
\begin{align} \label{eq:sumqi_D}
Q^A=\sum_{i=1}^{\infty}Q^A_i \in\R^{n\times n} 
\end{align}
with
\begin{align*}
    Q_i^A=\int_0^\infty\dots\int_0^\infty \bar{Q}^A_{i}(t_1,\dots, t_i) \left(\bar{Q}^A_{i}(t_1,\dots, t_i)\right)^T dt_1\dots dt_i,
\end{align*}
where $\bar{Q}^A_i$ as defined in \cref{eq:QiAsbar}, assuming that all integrals exist and the infinite series converge.
\end{definition}
Using the same arguments as in Section \ref{subsubsec:3-2},  one can show that $Q^A$ satisfies a (standard) Lyapunov equation which depends on the reachability Gramian $P$.
\begin{theorem}
    \label{thm:thm3}
     Let a BQO system \eqref{eq:BQO} with a stable matrix $A$ be given. Assume that the reachability Gramian $P$ and the observability Gramian $Q^A$ of the BQO system \cref{eq:BQO} defined as in (\ref{eq:sumpi}) and (\ref{eq:sumqi_D}) exist. Then, the Gramian $Q^A$ satisfies the $P$-dependent Lyapunov equation
        \begin{equation}
        A^TQ^A+Q^AA+\sum_{j=1}^pM_jPM_j + C^TC=0.\label{eq:thm3}
    \end{equation}
    \end{theorem}
\begin{remark}    
   For bilinear systems where $M_1=\dots=M_p=0$, the Lyapunov equation \cref{eq:thm3} for the observability Gramian $Q^A$ simplifies to the standard Lyapunov equation, $A^TQ+QA+C^TC=0$.
The solution to this equation is the observability Gramian for a linear system. However, this differs from the Lyapunov equation \cref{eq:bil_lyap_eq2}, which yields the observability Gramian in the bilinear case. That is, with the primal system \eqref{eq-dual-different}, we also obtain an alternative observability Gramian formulation for bilinear systems. The case for LQO systems is automatically included as before in \Cref{rem:gen_standard}.    \end{remark}
   \begin{remark}
        It is well known that (see, e.g., \cite{LanT85}) when $P$ is symmetric positive semidefinite matrix and $A$ is stable, the (standard) Lyapunov equation $A^TQ^A+Q^AA+\sum_{j=1}^p M_jPM_j+C^TC=0$ has a unique symmetric positive semidefinite solution $Q^A$. 
        Furthermore, if $(A^T,C^T)$ is reachable (equivalently if $(A,C)$ is observable), then $Q^A$ is positive definite.
    \end{remark}

\subsubsection{Observability Gramian using choice \cref{eq-dual-mix}}\label{subsubsec:3-4}
Similar to Subsections \ref{subsubsec:3-2} and \ref{subsubsec:3-3}, we can also derive the observability Gramian for a BQO system using the \emph{mixed} primal system formulation \cref{eq-dual-mix}. We will denote the resulting gramian by $Q^M$. {The derivation follows analogous to the derivation of $Q^S$ in \Cref{subsubsec:3-2} whereby we use $\phi_kN_k$ instead of $N_k$ here.}  Then, the resulting Gramian $Q^M$ solves a slightly different generalized Lyapunov equation. {This Lyapunov equation stated in the following theorem can be proven similar to \Cref{thm:thm2}.}
\begin{theorem}
    \label{thm:thm3a}
     Let a BQO system \eqref{eq:BQO} with a stable matrix $A$ be given. Assume that the reachability Gramian $P$ defined as in (\ref{eq:sumpi}) and the observability Gramian $Q^M$ of the BQO system \cref{eq:BQO} exist. Then the Gramian $Q^M$ satisfies the generalized $P$-dependent Lyapunov equation
        \begin{equation} 
        A^TQ^M+Q^MA+\sum_{k=1}^m \phi_k^2N_k^TQ^MN_k+\sum_{j=1}^pM_jPM_j + C^TC=0.\label{eq:thm3a}
    \end{equation} 
\end{theorem}
    In the following, we will mainly focus on the observability Gramians $Q^S$ and $Q^A$, which are identical to $Q^M$ for $\phi_k=1$ and $\phi_k=0$, for all $k=1,\dots,m$, respectively. An investigation of suitable choices for the parameters $\phi_k$ is left for future research.
\section{Existence of the Gramians}
\label{sec:5}
We will first prove the existence of the different Gramians introduced in Section \ref{sec:3}. 
Then, we will show that $Q^S$ and $Q^P$ are indeed equal by proving that the associated Lyapunov equation has a unique solution.

Recall that $A$ is stable if there exist $\beta >0$ and $0<\alpha \leq -\max_i(\text{Re}(\lambda_i(A)))$ such that $\norm{e^{At}}\leq \beta e^{-\alpha t}$ for $t\geq 0$ \cite{adrianova_1995,willems70}. The constants $\alpha$ and $\beta$ determine whether the different Gramians from Section \ref{sec:3} exist. For each of the Gramians to exist, a slightly different condition needs to hold.    
    
   \begin{theorem}\label{thm:PQexist} Let a BQO system \eqref{eq:BQO} with a stable matrix $A$ be given. Let the stability parameters $\beta >0$ and $0<\alpha \leq -\max_i(Re(\lambda_i(A)))$ be given such that $\norm{e^{At}}\leq \beta e^{-\alpha t}$ for $t\geq 0$. Then,         
            \begin{enumerate}
                \item[(a)] the reachability Gramian $P$ in \cref{eq:sumpi} exists if \\
                ~\hspace*{2cm}{$\Gamma_P := \norm{\mathcal{N}}^2 < 2\alpha/\beta^2,$ } 
                
                \item[(b)]  the observability Gramians $Q^S$ in \cref{eq:sumqi} and 
                $Q^P$ in \cref{eq:sumqi_P} exist if \\
                ~\hspace*{2cm}{$\Gamma_{Q^S} = \Gamma_{Q^P}:=\max\{ \Gamma_P , \norm{\mathcal{N}_T}^2, \norm{\mathcal{M}}^2\}< 2\alpha/\beta^2,$}     
                
                \item[(c)] the observability Gramian $Q^A$ in \cref{eq:sumqi_D} exists if \\
                ~\hspace*{2cm}{$\Gamma_{Q^A}:=\max\{\Gamma_P,\norm{\mathcal{M}}^2\}< 2\alpha/\beta^2.$}            
            \end{enumerate}
                  
    \end{theorem}

    \begin{proof}
        The proof of (a) for the existence of $P$ can be found, e.g.,  in \cite{zhang_2002}.  Thus we focus on the existence conditions for observability Gramians, starting with $Q^S$.
        Recall the equations \cref{eq:sumpi,eq:sumqi}, and 
        \begin{align*}
            \bar{P}_1 &= e^{At_1}B, &\bar{Q}^S_1 = e^{A^Tt_1}C^T,\\
            \bar{P}_i &= e^{At_i}\mathcal{N}(I_m\otimes \bar{P}_{i-1}), &\bar{Q}^S_i = e^{A^Tt_i}[\mathcal{N}_T(I_m\otimes \bar{Q}^S_{i-1})\quad\mathcal{M}(I_p\otimes \bar{P}_{i-1})],
        \end{align*}
        for $i>1$.
        As already noted in \cite{zhang_2002}, it holds that
        \begin{subequations}       
        \label{eq:norm_pi}
        \begin{align}
            \norm{\bar{P}_i\bar{P}_i^T} &= \norm{e^{At_i}\mathcal{N}(I_m\otimes \bar{P}_{i-1}) (I_m\otimes \bar{P}_{i-1}^T)\mathcal{N}^Te^{A^Tt_i}}\\
            &\leq \norm{e^{At_i}} \norm{\mathcal{N}}\norm{(I_m\otimes \bar{P}_{i-1}) (I_m\otimes \bar{P}_{i-1}^T)} \norm{\mathcal{N}^T}\norm{e^{At_i}}\\
            &\leq \beta^2 e^{-2\alpha t_i} \Gamma_{Q^S}  
            \norm{\bar{P}_{i-1}\bar{P}_{i-1}^T} \label{eq:4.1c} \\
            &\leq  \beta^{2i}e^{-2\alpha(t_i+t_{i-1}+\dots+t_1)} \Gamma_{Q^S}^{i-1} 
            \norm{BB^T}, \label{eq:4.1d}
        \end{align}        
        \end{subequations}
 {where, for the derivation of \cref{eq:4.1c}, we used the properties of the Kronecker product, the definition of $\Gamma_{Q^S}$, and the estimate for $\norm{e^{At}}$. Then, \cref{eq:4.1d} follows with repeated substitutions for $i$.} 
        In a similar fashion, we obtain 
        \begin{align*}
            \norm{\bar{Q}^S_1\left(\bar{Q}_1^S\right)^T} &= \norm{e^{A^Tt_1}C^TCe^{At_1}} \leq \beta ^2 e^{-2\alpha t_1}\norm{C^TC},
            \end{align*} 
            and
            \begin{align} \notag
            \norm{\bar{Q}_i^S\left(\bar{Q}_i^S\right)^T} &= \Big\Vert e^{A^Tt_i}\left[\mathcal{N}_T(I_m\otimes \bar{Q}_{i-1}^S) \quad \mathcal{M}(I_p\otimes \bar{P}_{i-1})\right] 
            \begin{bmatrix}(I_m\otimes \left(\bar{Q}_{i-1}^S\right)^T) \mathcal{N}_T^T \\ (I_p\otimes \bar{P}_{i-1}^T)\mathcal{M}^T\end{bmatrix}e^{At_i}\Big\Vert\\
            &=\Big\Vert e^{A^Tt_i}\left[\mathcal{N}_T(I_m\otimes \bar{Q}_{i-1}^S\left(\bar{Q}_{i-1}^S\right)^T) \mathcal{N}_T^T +\mathcal{M}(I_p\otimes \bar{P}_{i-1}\bar{P}_{i-1}^T)\mathcal{M}^T\right] e^{At_i}\Big\Vert  \notag \\
            &\leq \beta^2e^{-2\alpha t_i} \left(\norm{\mathcal{N}_T}^2 \norm{\bar{Q}_{i-1}^S\left(\bar{Q}_{i-1}^S\right)^T}+\norm{\mathcal{M}}^2 \norm{\bar{P}_{i-1}\bar{P}_{i-1}^T} \right)\label{eq:Qs0}\\
            &\leq \beta^2e^{-2\alpha t_i} \Gamma_{Q^S} 
            \left( \norm{\bar{Q}_{i-1}^S\left(\bar{Q}_{i-1}^S\right)^T}+ \norm{\bar{P}_{i-1}\bar{P}_{i-1}^T} \right)\label{eq:Qs1}\\
            & \leq \beta^2e^{-2\alpha t_i} \Gamma_{Q^S} 
            \norm{\bar{Q}_{i-1}^S\left(\bar{Q}_{i-1}^S\right)^T} + \beta^{2i}e^{-2\alpha(t_i+t_{i-1}+\dots+t_1)} \Gamma_{Q^S}^{i-1} 
            \norm{BB^T},  \label{eq:Qs2}  
            \end{align}
            where \cref{eq:Qs1} results from using the definition of $\Gamma_{Q^S}$ in \cref{eq:Qs0} and \cref{eq:Qs2} follow from using \cref{eq:norm_pi} in
            \cref{eq:Qs1}.
    Hence,
    \begin{align*}
      \norm{\bar{Q}^S_i\left(\bar{Q}_i^S\right)^T}       \leq& \beta^{2i}e^{-2\alpha(t_i+t_{i-1}+\dots+t_1)} \Gamma_{Q^S}^{i-1} 
      \notag \left((i-1)\norm{BB^T}+\norm{C^TC}\right).
        \end{align*}
This yields
\begin{align*}
    \norm{Q^S}\leq& \sum_{i=1}^\infty \int_0^\infty \dots \int_0^\infty \norm{\bar{Q}^S_i\left(\bar{Q}_i^S\right)^T}dt_1\dots dt_i\\
    \leq& \sum_{i=1}^\infty \int_0^\infty \dots \int_0^\infty \beta^{2i}e^{-2\alpha(t_i+t_{i-1}+\dots+t_1)} \Gamma_{Q^S}^{i-1}\\ 
     & \qquad \qquad \cdot \left((i-1)\norm{BB^T}+\norm{C^TC}\right) dt_1\dots dt_i\\
    \leq& \frac{1}{\Gamma_{Q^S}} 
    \sum_{i=1}^\infty \left( \frac{\beta^2\Gamma_{Q^S}}
    {2\alpha}\right)^i \left((i-1)\norm{BB^T}+\norm{C^TC}\right).
\end{align*}
        Since 
        $\beta^2\Gamma_{Q^S}/(2\alpha) <1$ (or, equivalently $\Gamma_{Q^S} < 2\alpha/\beta^2$) by the assumption, 
        this term is finite due to the convergence of the geometric series, completing the proof for the existence of $Q^s$. The statement (c) for the existence of $Q^A$ follows analogously to the proof for $Q^S$.
        
        For the proof of the existence of $Q^P$, we first recall the definition of the bilinear observability Gramian from \cref{eq:sumqi_B} and the expressions
        \begin{align*}
            \bar{Q}_1^B=e^{A^Tt_1}C^T, \quad \bar{Q}_i^B=e^{A^Tt_i}\mathcal{N}_T (I_m\otimes \bar{Q}_{i-1}^B).
        \end{align*}
        Analogously to before, we have
        \begin{align*}
            \norm{\bar{Q}_1^B\left(\bar{Q}_1^B\right)^T}&\leq \beta^2e^{-2\alpha t_1} \norm{C^TC},~\mbox{and}\\
            \norm{\bar{Q}_i^B\left(\bar{Q}_i^B\right)^T}&\leq \beta^{2i}e^{-2\alpha (t_i+\dots +t_1)}  \Gamma_{Q^P}^{i-1}  
            \norm{C^TC},
        \end{align*}
        which lead to 
        \begin{align*}
            \norm{Q^B}\leq& \sum_{i=1}^\infty \int_0^\infty \dots \int_0^\infty \norm{\bar{Q}^B_i\left(\bar{Q}_i^B\right)^T}dt_1\dots dt_i\\
            \leq& \sum_{i=1}^\infty \int_0^\infty \dots \int_0^\infty \beta^{2i}e^{-2\alpha(t_i+t_{i-1}+\dots+t_1)} \Gamma_{Q^P}^{i-1} 
            \norm{C^TC} dt_1\dots dt_i\\
            \leq& \frac{1}{\Gamma_{Q^P}} 
            \sum_{i=1}^\infty \left( \frac{\beta^2\Gamma_{Q^P}} 
            {2\alpha}\right)^i \norm{C^TC}.
        \end{align*} 
        Hence, $Q^B$ exists under the assumption on $\Gamma_{Q^P}$. Next, recall \cref{eq:sumqi_P}, and 
        \begin{align*}
            \bar{Q}_{1,j}= e^{A^Ts_1}\mathcal{M}(I_p\otimes \bar{P}_j), \quad \bar{Q}_{i,j}= e^{A^Ts_i}\mathcal{N}_T(I_m\otimes \bar{Q}_{i-1,j}).
        \end{align*}
        Now, we estimate the following terms as before:
        \begin{align*}
    \norm{\bar{Q}_{1,j}\left(\bar{Q}_{1,j}\right)^T}&\leq \beta^2e^{-2\alpha s_1} \Gamma_{Q^P} 
            \norm{\bar{P}_j\bar{P}_j^T}, \\
        \norm{\bar{Q}_{i,j}\left(\bar{Q}_{i,j}\right)^T}&\leq \beta^{2i}e^{-2\alpha (s_i+\dots +s_1)} \Gamma_{Q^P}^i  
            \norm{\bar{P}_j\bar{P}_j^T}.
        \end{align*}
        Then, using \cref{eq:norm_pi} for $i,j=1,2\dots$, we obtain
        \begin{align*}            
    \norm{\bar{Q}_{i,j}\left(\bar{Q}_{i,j}\right)^T}&\leq \beta^{2i}e^{-2\alpha (s_i+\dots +s_1)} \Gamma_{Q^P}^i 
        \beta^{2j}e^{-2\alpha(t_j+\dots+t_1)}\Gamma_{Q^P}^{j-1} 
            \norm{BB^T}.
        \end{align*}
        These expressions all together yield
        \begin{align*}
            \norm{\sum_{i,j=1}^\infty Q_{i,j}}\leq& \sum_{i,j=1}^\infty \int_0^\infty \dots \int_0^\infty \norm{\bar{Q}_{i,j}\left(\bar{Q}_{i,j}\right)^T}ds_1\dots ds_idt_1\dots dt_j\\
            \leq& \sum_{i=1}^\infty \int_0^\infty \dots \int_0^\infty \beta^{2i}e^{-2\alpha (s_i+\dots +s_1)}  \Gamma_{Q^P}^i  
            ds_1\dots ds_i \\
            &\qquad \cdot \int_0^\infty \dots \int_0^\infty\beta^{2j}e^{-2\alpha(t_j+\dots+t_1)}\Gamma_{Q^P}^{j-1} 
            \norm{BB^T} dt_1\dots dt_j\\
            \leq& \sum_{i=1}^\infty \left( \frac{\beta^2\Gamma_{Q_P}} 
            {2\alpha}\right)^i \cdot \frac{1}{\Gamma_{Q^P}} 
            \sum_{j=1}^\infty \left( \frac{\beta^2\Gamma_{Q^P}} 
            {2\alpha}\right)^j \norm{BB^T}.
        \end{align*}
        So, both of the two summands in \cref{eq:sumqi_P}  
        exists, thus also the observability Gramian $Q^P$, completing the proof. 
      \end{proof}

\begin{remark}
Note that,  
since 
        \begin{align*}
\norm{\mathcal{N}}^2=\norm{\mathcal{N}\mathcal{N}^T}\leq \sum_{k=1}^m\norm{N_k}^2,
        \end{align*}
in part (a) of \Cref{thm:PQexist},
instead of $\Gamma_P = 
\norm{\mathcal{N}}^2$, one can choose $\Gamma_P  
= \sum_{k=1}^m\norm{N_k}^2$ to obtain a similar bound as in {\cite{zhang_2002,gray_energy_1998,BenG24}} for the Gramian $P$. 
Analogously, one can replace $\norm{\mathcal{N}_T}^2$ and $\norm{\mathcal{M}}^2$, with 
$\sum_{k=1}^m\norm{N_k}^2$ and 
$\sum_{j=1}^p\norm{M_j}^2$, respectively in parts (b) and (c), which will be a looser upper bound.
    \end{remark}
Next, we focus on the existence and uniqueness of a symmetric positive semidefinite solution of the Lyapunov equations related to $P$, $Q^S$ and $Q^P$, and will show that under these conditions $Q^S=Q^P$ holds.

Recall that the reachability Gramian $P$ can be expressed as $P = \lim_{\ell \rightarrow \infty} \hat{P}_\ell$ where $\hat P_\ell$ solves the generalized Lyapunov equation \cref{eq:P_fixit}
    \[
    A\hat P_\ell +\hat P_\ell A^T + \sum_{k=1}^m N_k \hat P_{\ell-1}N_k^T + BB^T = 0 \quad \text{~for~} \ell \geq 2
    \]
    and $A\hat P_1 + \hat P_1 A^T +BB^T = 0$. Thus, $\hat P_\ell$ can be computed via a fixed-point iteration. Similarly, for the observability Gramian $Q^S = \lim_{\ell \rightarrow \infty} \hat{Q}^S_\ell$, the matrices $\hat Q^S_\ell$ can also be computed via a fixed-point iteration based on \eqref{eq:QS_hat}. In order to prove the existence and uniqueness of the solution of the Lyapunov equations for $P$ and $Q^S$, we will show that the above mentioned fixed-point iterations converge under a reasonable condition on $\mathcal{N}, \mathcal{N}_T, \mathcal{M}$ when $A$ is stable. 
    As a preparatory step, we recall 
    the following theorem  from \cite{damm_direct_2008,BenD11} on the uniqueness of the solutions to the generalized Lyapunov equations. {See also \cite[Theorem 3.6.1]{damm_rational_2004} for further equivalent conditions.} 
    \begin{theorem}{\cite[Theorem 4.1]{damm_direct_2008}\cite[Theorem 3.9]{BenD11}}\label{thm:damm}
        Consider $A\in\R^{n\times n}$ and linear operators $\mathcal{L}_A,\Pi:\R^{n\times n} \to \R^{n\times n}$ so that $\mathcal{L}_A(X):=AX+XA^T$ and $\Pi$ is nonnegative, i.e., $\Pi(X)\geq 0$ if $X\geq 0$. The following are equivalent:          
           \begin{enumerate}
            \item For every $Y>0$, there exists an  $X>0$ such that  $\mathcal{L}_A(X)+\Pi(X)=-Y$.
            \item There exist $Y>0$  and $X>0$ such that $\mathcal{L}_A(X)+\Pi(X)=-Y$.
            \item {There exist $Y\geq 0$ such that $(A,Y)$ is reachable and $X>0$ such that $\mathcal{L}_A(X)+\Pi(X)=-Y$.}
            \item $\sigma(\mathcal{L}_A+\Pi)\subset \C_-$.
            \item $\sigma(\mathcal{L}_A)\subset\C_-$ and $\rho(\mathcal{L}_A^{-1}\Pi)<1$.
        \end{enumerate}    
    \end{theorem}
    {The work \cite{BenD11} also states that the stability of $A$, the reachability of $(A,B)$, and some sufficiently small $N_k$'s guarantee a positive definite solution of the associated Lyapunov equation.}
    
    With this result, we can state and prove a theorem concerning the existence and uniqueness of the solution of the Lyapunov equations \eqref{eq:thm1a} for $P$  and \eqref{eq:thm1b} for $Q^S$.

    \begin{theorem}\label{thm:Layp-exist}
    Given is the BQO system \cref{eq:BQO} with stable $A$ as in \Cref{thm:PQexist}.
        \begin{itemize}
        \item[(a)] If $\Gamma_P <2\alpha/\beta^2$, then
        there exists a unique symmetric positive semidefinite solution $P\in\R^{n\times n}$ of \cref{eq:thm1a}, i.e., $AP+PA^T+\sum_{k=1}^mN_kPN_k^T+BB^T=0$. 
        Furthermore, if $(A,B)$ is reachable, then $P$ is positive definite.
        \item[(b)] If $\Gamma_{Q^S} <2\alpha/\beta^2$, then there exists a unique symmetric positive semidefinite solution  $Q^S\in\R^{n\times n}$ of \cref{eq:thm1b}, i.e., $A^TQ^S+Q^SA+\sum_{k=1}^mN_k^TQ^SN_k+\sum_{j=1}^p M_jPM_j+C^TC=0$.  
        Furthermore, if $(A^T,C^T)$ is reachable (equivalently if $(A,C)$ is observable), then $Q^S$ is positive definite. 
        \end{itemize}
    \end{theorem}
    \begin{proof}
        We start with proving the statements for $P.$ Let $\mathcal{L}_A(X) = AX+XA^T$. 
        The fact $\sigma(\mathcal{L}_A)\subset \C_-$ follows directly from the stability of $A$.        
        The choice $\Pi(X)=\sum_{k=1}^mN_kXN_k^T$
        guarantees $\Pi(X)\geq 0$ for $X\geq 0$.         
        The equation $\mathcal{L}(X) + \Pi(X) = 0$ can be rewritten as a fixed-point equation $X =  -\mathcal{L}_A^{-1}(\Pi(X))$, which has a unique solution $X$ if  $\rho(\mathcal{L}_A^{-1}\Pi)<1$. Since $\rho(\mathcal{L}_A^{-1}\Pi)\leq\norm{\mathcal{L}_A^{-1}\Pi}$, it is sufficient to show that $\norm{\mathcal{L}_A^{-1}\Pi}<1$ holds under the assumptions. Now, for $X,Y\in \R^{n\times n}$ such that $\mathcal{L}_A^{-1}(\Pi(X))=Y$, we have
        \[
            Y=-\int_{0}^\infty e^{At}\left(\sum_{k=1}^m N_kXN_k^T \right) e^{A^Tt}dt.
            \]
            Using the techniques from the proof of \Cref{thm:PQexist}, we obtain
            \begin{align*}
            \norm{Y}\leq& \int_{0}^\infty \norm{e^{At}}\norm{\sum_{k=1}^m N_kXN_k^T } \norm{e^{A^Tt}}dt \\
            \leq& \frac{\beta^2}{2\alpha}\norm{\mathcal{N}}^2\norm{X}.
        \end{align*}
        Hence, it holds that $\norm{\mathcal{L}_A^{-1}\Pi}\leq \frac{\beta^2}{2\alpha}\norm{\mathcal{N}}^2$. Due to the assumption on $\Gamma_P$, this term is smaller than $1$ and thus $\rho(\mathcal{L}_A^{-1}\Pi)<1$ holds.        
        Using the facts $\sigma(\mathcal{L}_A)\subset\C_-$ and $\rho(\mathcal{L}_A^{-1}\Pi)<1$ and Theorem \ref{thm:damm}, it follows that $\sigma(\mathcal{L}_A+\Pi)\subset\C_-$. Hence, the equation $(\mathcal{L}_A+\Pi)(P)=-BB^T$ has a unique solution which can be computed by a fixed point iteration \cref{eq:P_fixit}. 

        The solution $P$ is  symmetric since $0=(\mathcal{L}_A(P)+\Pi(P)+BB^T)^T=\mathcal{L}_A(P^T)+\Pi(P^T)+BB^T$.  
        It is well known that, see, e.g., \cite{LanT85}, assuming $A$ is stable, the solution $P$ of $AP+PA^T+F=0$ is symmetric positive semidefinite 
        if $F\geq 0$.  
        With this result in mind, we investigate the recursive definition of $P$ in \cref{eq:P_rec}. We see directly that $P_1 \geq 0$
        if $BB^T\geq 0$.  
        Hence, we can use a factorization $P_1=L_1L_1^T$ in the equation for $P_2$ to obtain
        \begin{align*}
            0&=AP_2+P_2A^T+\sum_{k=1}^mN_kP_{1}N_k^T \\
            &= AP_2+P_2A^T+ \begin{bmatrix} N_1L_1 & \cdots & N_mL_1 \end{bmatrix} \begin{bmatrix} N_1L_1 & \cdots & N_mL_1 \end{bmatrix}^T.
        \end{align*}
        Thus, $P_2$ is symmetric positive semidefinite if $BB^T + \sum_{k=1}^m N_k P_1N_k^T$ is symmetric positive semidefinite and, by induction, so are $P_i$ and $P=\sum_{i=1}^\infty P_i$.
        Thus, if $BB^T$ is positive semidefinite, then $P_i$ and $P$ are positive semidefinite. {If now additionally $(A,B)$ is reachable, $P$ is even positive definite due to \cref{thm:damm}}.
        The results for $Q^S$ follow analogously with $\mathcal{L}_A = A^TX+XA$ and $\Pi(X) = \sum_{k=1}^mN_k^TXN_k$.  
    \end{proof}
    
    As we showed in \Cref{subsubsec:3-2}, the Gramians  $Q^S$ and $Q^P$ solve the same generalized Lyapunov equation~\cref{{eq:thm1b}}, which has a unique solution due to \cref{thm:Layp-exist}. This immediately yields the following result.
    \begin{corollary}
        Let $Q^S$ and $Q^P$ exist under the conditions of \Cref{thm:PQexist}. Then, they are equal to each other.  
    \end{corollary}

    \section{Approximation and truncation of Gramians}
    \label{sec:6}
The computation of the reachability Gramian $P$ and the observability Gramian $Q^S$ (or $Q^P$) require solving the generalized Lyapunov equations \cref{eq:thm1a,eq:thm1b}, namely
\begin{align*}
    AP + PA^T + \sum_{k=1}^m N_kPN_k^T + BB^T &= 0,\\
    A^TQ^S + Q^SA + \sum_{k=1}^m N_k^TQ^SN_k + \sum_{j=1}^pM_jPM_j + C^TC &= 0.
\end{align*}
Even though there have been many advances in methods to compute  (low-rank)
solutions of these generalized Lyapunov equations (see, e.g, \cite{BenB13,ShaSS16}), this is still a demanding task. In order to overcome this bottleneck, in \cite{BenGR17} the so-called truncated Gramians for bilinear systems have been introduced. Their advantage is that they can be computed by solving only standard Lyapunov equations. In particular, for the reachability Gramian $P$, it is suggested to consider only the first two summands $P_1$ and $P_2$ of the infinite sum $P = \sum_{\ell=1}^\infty P_\ell$ \eqref{eq:P_rec}; in other words, the truncated reachability Gramian, $P_T$, is defined as
\begin{equation}\label{eq:P_T}
P_T:=P_1+P_2.
\end{equation}
It is shown that $P_T$ satisfies the standard Lyapunov equation
\begin{equation} \label{eq:PTlyap}
AP_T+P_TA^T + \breve{B}\breve{B}^T = 0 \qquad \text{with} \quad \breve{B}\breve{B}^T = \sum_{k=1}^mN_kP_1N_k^T+BB^T,
\end{equation}
where $P_1$ is the solution of $AP_1 + P_1A^T+BB^T = 0$.
These results also follow easily from the proof of \Cref{thm:thm1}. Since $P_1$ is symmetric positive semidefinite, it suffices to compute its Cholesky factor $L_1$ such that $P_1 = L_1L_1^T$. Then in \cref{eq:PTlyap}, we have $\breve{B} = \begin{bmatrix} N_1L_1 & \cdots & N_mL_1 & B\end{bmatrix}$ . Low-rank modifications follow easily. The algorithm for computing $P_T$ is summarized in \cref{alg1}.

\begin{algorithm}
\caption{Computation of $P_T$}\label{alg1}
\begin{algorithmic}[1]
\Require $A, B, N_k$ as in \eqref{eq:BQO}
\Ensure Truncated Gramian $P_T$
\State Solve $AP_1+P_1A^T+BB^T=0$ for $P_1$.
\State Solve $AP_T+P_TA^T+\sum_{k=1}^m N_kP_1N_k^T+BB^T=0$ for $P_T$.
\end{algorithmic}
\end{algorithm}
Thus, instead of solving the generalized Lyapunov equation \eqref{eq:thm1a} for $P$ (e.g., via the fixed-point iteration \eqref{eq:P_fixit} by a series of standard Lyapunov equations), solving just two standard Lyapunov equations is required in order to determine its approximation $P_T.$ In case one needs to include additional summands $P_j, j = 3, 4, \dots$ of $P$ in order to achieve a more accurate approximate truncated Gramian $P_T$, for each summand $P_j$ an additional standard Lyapunov equation needs to be solved.

Next, we investigate the observability Gramian $Q^S$ \eqref{eq:sumqi} based {on the dual system with choice \cref{eq-dual-standard}.}  
Due to \eqref{eq_Lyap_Qi}, 
    and the choice of $P_T$ in \cref{{eq:P_T}}, we choose the truncated approximation to 
    $Q^S$ as 
    \begin{equation} \label{eq:Q_ST} Q^S_T:=Q^S_1+Q^S_2+Q^S_3 \end{equation} 
    since this allows us to make full use of the already determined $P_T.$  
    In most cases, including only the first two terms $Q^S_1$ and $Q^S_2$ is insufficient for capturing the complete dynamics.
    The resulting algorithm for computing $Q^S_T$ is summarized in \cref{alg2}. Instead of solving the generalized Lyapunov equation \eqref{eq:thm1b} for $Q^S$, solving three standard Lyapunov equations is required to determine its approximation $Q_T$. In case, one needs to include additional summands $Q^S_j, j = 3, 4, \dots$, for each summand $Q^S_j$ an additional standard Lyapunov equation needs to be solved.

\begin{algorithm}
\caption{Computation of $Q^S_T$}\label{alg2}
\begin{algorithmic}[1]
\Require $A, C,  N_k, M_j$ as in \eqref{eq:BQO} and $P_1$ and $P_T$ computed by Algorithm \ref{alg1}
\Ensure Truncated Gramian $Q^S_T$
\State Solve $A^TQ^S_1+Q^S_1A+C^TC=0$ for $Q^S_1$.
\State Solve $A^T\hat{Q}^S+\hat{Q}^SA+\sum_{k=1}^m N_k^TQ^S_1N_k+\sum_{j=1}^pM_jP_1M_j+C^TC=0$ \\ \qquad for $\hat{Q}^S=Q_1^S+Q_2^S$.
\State Solve $A^TQ_T^S+Q_T^SA+\sum_{k=1}^m N_k^T\hat{Q}^SN_k+\sum^p_{j=1}M_jP_TM_j+C^TC=0$ for $Q_T^S.$
\end{algorithmic}
\end{algorithm}

Recall that $Q^S = Q^P$ holds for $Q^P$ as in \eqref{eq:sumqi_P}. Using the infinite sum \eqref{eq:sumqi_P}, we obtain a different truncated observability Gramian $Q_T^P:= Q_1^B+Q_2^B+ Q_{1,1}+Q_{1,2}+Q_{2,1}+Q_{2,2}$. This has already been discussed in \cite{padhi_2024} and is summarized in \Cref{alg3}.   
\begin{algorithm}
\caption{Computation of $Q^P_T$}\label{alg3}
\begin{algorithmic}[1]
\Require $A, C,  N_k, M_j$ as in \eqref{eq:BQO} and  $P_T$ computed by Algorithm \ref{alg1}
\Ensure Truncated Gramian $Q^P_T$
\State Solve $A^T\hat{Q}^P+\hat{Q}^PA+\sum_{j=1}^pM_jP_TM_j+C^TC=0$ for $\hat{Q}^P$.
\State Solve $A^TQ_T^P+Q_T^PA+\sum_{k=1}^m N_k^T\hat{Q}^PN_k+\sum_{j=1}^pM_jP_TM_j+C^TC=0$ for $Q_T^P$.
\end{algorithmic}
\end{algorithm}
Thus, only two standard Lyapunov equations need to be solved to determine $Q_T^P$, one less than for the computation of $Q_T^S$. Therefore, Algorithm \ref{alg3} will always be faster than Algorithm \ref{alg2}.

Finally, we consider the observability Gramian $Q^A$ in \eqref{eq:sumqi_D}. This alternative Gramian $Q^A$ is the solution of a standard Lyapunov equation \eqref{eq:thm3} so there is no need to truncate it. However, using $P_T$ instead of $P$ in \eqref{eq:thm3} yields a solution which we refer to as $Q_T^A$ and $Q_T^A \neq Q^A$.
The resulting algorithm to determine $Q_T^A$ is summarized in Algorithm \ref{alg4}. There are no savings with respect to computational efforts. We observe that $Q^A_T=\hat{Q}^P$ in step 1 of \Cref{alg3}.

 \begin{algorithm}
\caption{Computation of $Q^A_T$}\label{alg4}
\begin{algorithmic}[1]
\Require $A, C, M_j$ as in \eqref{eq:BQO} and  $P_T$ computed by Algorithm \ref{alg1}
\Ensure Truncated Gramian $Q^A_T$
\State Solve $A^TQ^A_T+Q^A_TA+\sum_{j=1}^pM_jP_TM_j+C^TC=0$ for $Q^A_T$.
\end{algorithmic}
\end{algorithm}

    \section{Balanced truncation of BQO systems}\label{sec:7}
Now that we have defined the Gramians for BQO systems, we can apply
balanced truncation in a similar way as in the linear case. In this process, the original system is transformed such that the reachability and the observability Gramians are diagonal and equal.
By keeping only the most important states (i.e., easiest to reach and easiest to observe), we obtain a reduced-order model. Theory for linear balancing and balanced truncation can be found, e.g., in \cite{antoulas_book_2005,BenB17,BreS21,GugA04}, or \cite[Section 7 and 8]{benner_modellreduktion_2024}. 

    To do this, we first need a suitable definition of (un)reachable/(un)observable states. 
    The energy functionals for nonlinear systems are no longer quadratic as in the linear case (see, e.g., \cite[Section 2]{benner_lyapunov_2011}). The reachability functional for a nonlinear system $\dot{x}(t)=f(x(t))+g(x(t))u(t), \, y(t)=h(x(t))$ is defined as
    \begin{align}
        L_c(x_0)= \min_{\substack{u\in L^2(-\infty,0], \\ x(-\infty)=0, \, x(0)=x_0}}
        \frac{1}{2}\int_{-\infty}^0\norm{u(t)}^2dt \label{eq:contr_funct}
    \end{align}
    with $x_0\in\R^n$ and the observability functional is defined as
    \begin{align}
        L_o(x_0)=\max_{\substack{u\in L^2[0,\infty)], \norm{u}_{L^2}\leq\alpha \\ x(0)=x_0,\, x(\infty)=0}}
        \frac{1}{2}\int_0^\infty\norm{y(t)}^2dt \label{eq:obs_funct}
    \end{align}
    with a fixed parameter $\alpha\geq 0$; see, e.g., \cite[Definition 2.1]{gray_energy_1998}. We call states $x_0\in\R^n$ with $L_c(x_0)=\infty$ unreachable and with $L_o(x_0)=0$ unobservable, respectively. The choice of $\alpha$ in \cref{eq:obs_funct} allows for different approaches to study observability of BQO systems. For $\alpha >0$, we investigate the so-called \emph{homogeneous BQO system} corresponding to \cref{eq:BQO}, namely
    \begin{subequations}\label{eq:BQO_hom}
    \begin{align}
    \dot{x}(t) &= Ax(t) + \sum_{k=1}^m N_k x(t) u_k(t),\qquad x(0)=0, \label{eq:BQO_hom_1}\\
    y(t)&= Cx(t)+\begin{bmatrix}
        x(t)^TM_1x(t) \\ \vdots \\ x(t)^TM_px(t)
    \end{bmatrix}.\label{eq:BQO_hom_2}
    \end{align}
    \end{subequations}
    This is similar to the homogeneous bilinear system in \cite[p.695]{benner_lyapunov_2011}.
    Another approach follows \cite[Definition 3.1]{scherpen_balancing_1993} where $\alpha=0$. This leads to an input $u\equiv 0$, so we do not need to rely on the homogeneous system \cref{eq:BQO_hom} and can work with the original system \cref{eq:BQO}. These different approaches for the observability functional are investigated in the next theorem, which extends
    \cite[Theorem 3.3.4]{padhi_2024} to the MIMO setting and to the newly proposed Gramians. 
    
    \begin{theorem}\label{thm:nullPQ} 
        \begin{enumerate}[label=(\alph*)]
            \item Consider the BQO system \cref{eq:BQO} and let $P\geq 0$ exist as a solution to \cref{eq:thm1a}. If $x_0\in\ker(P)$, then $x_0$ is unreachable regarding the reachability functional in \cref{eq:contr_funct}.
            \item Consider a homogeneous BQO system \cref{eq:BQO_hom} and let $P\geq 0$ and $Q^S\geq 0$ exist and be solutions to \cref{eq:thm1a,eq:thm1b}. If $P>0$ and $x_0\in\ker(Q^S)$, then $x_0$ is unobservable regarding the observability functional in \cref{eq:obs_funct} with $\alpha\geq0$.
            \item Consider the BQO system \cref{eq:BQO} with zero input $u\equiv 0$ and let $P\geq 0$ and $Q^S\geq 0$ exist and be solutions to \cref{eq:thm1a,eq:thm1b}. If $P>0$ and $x_0\in\ker(Q^S)$, then $x_0$ is unobservable regarding the observability functional in \cref{eq:obs_funct} with $\alpha=0$.
            \item Consider the BQO system \cref{eq:BQO} with zero input $u\equiv 0$ and let $P\geq 0$ and $Q^A\geq 0$ exist and be solutions to \cref{eq:thm1a,eq:thm3}. If $P>0$ and $x_0\in\ker(Q^A)$, then $x_0$ is unobservable regarding the observability functional in \cref{eq:obs_funct} with $\alpha=0$.
        \end{enumerate}
    \end{theorem}
    \begin{proof}
        \begin{enumerate}[label=(\alph*)]
            \item This fact has already been proven for bilinear systems in \cite[Theorem 3.1]{benner_lyapunov_2011} and can be applied here directly.
            \item The proof of this statement follows the one of
            \cite[Theorem 3.1]{benner_lyapunov_2011}. 
            Let $x_0\in \ker(Q^S)$. Multiplying \cref{eq:thm1b} from both sides by $x_0^T$ and $x_0$, respectively, yields 
            \[0=x_0^T\left(\sum_{k=1}^mN_k^TQ^SN_k+\sum_{j=1}^pM_jPM_j+C^TC\right)x_0.\] Hence, $x_0^TC^TCx_0 =0$, which implies $\ker(Q^S) \subseteq \ker(C)$. 
            Similarly, we see that $N_k\ker(Q^S)\subseteq \ker(Q^S)$ and $M_j\ker(Q^S) \subseteq \ker(P)$ holds. Moreover, multiplying \cref{eq:thm1b} only from the right side with $x_0$ yields $A\ker(Q^S)\subseteq \ker(Q^S)$.            
            It follows {for the homogeneous system \cref{eq:BQO_hom}} that if $x(t)\in\ker(Q^S)$, then \[
            Q^S\dot{x}(t)=Q^SAx(t)+\sum_{k=1}^mQ^SN_kx(t)u_k(t)=0,
            \]
            i.e., $\dot{x}(t)\in\ker(Q^S)$. Thus, $\ker(Q^S)$ is invariant under the dynamics, i.e., $x_0\in\ker(Q^S) \Rightarrow x(t)\in\ker(Q^S)$. With the assumption $P>0$ and $M_j\ker(Q^S)\subseteq\ker(P)$, $M_j\ker(Q^S)=\{0\}$ holds. Due to this fact and the fact $\ker(Q^S) \subseteq \ker(C)$, it follows that $y(t)=0$ for $x(t)\in\ker(Q^S)$.
            \item This statement follows analogously to the proof of (b) above using $Bu\equiv 0$.
            \item Let $x_0\in\ker(Q^A)$. Multiplying \cref{eq:thm3} from both sides by $x_0^T$ and $x_0$, respectively, yields 
            \[
            0=x_0^T\left(\sum_{j=1}^pM_jPM_j+C^TC\right)x_0.
            \]
            As in the proof of (b) above we have $\ker(Q^A)\subseteq \ker(C),
            M_j\ker(Q^A)\subseteq \ker(P)$ and $A\ker(Q^A)\subseteq \ker(Q^A)$, but in general $N_k\ker(Q^A)\not\subseteq\ker(Q^A)$. Hence, the importance of $u\equiv 0$ becomes apparent because now $Q^A\dot{x}(t)=Q^AAx(t)=0$ holds for $x(t)\in\ker(Q^S)$. 
            The rest of the proof follows as before. 
        \end{enumerate}
    \end{proof}

   In contrast to \cite{padhi_2024}, Theorem \ref{thm:nullPQ} is formulated and proved here for the more general MIMO case.
    In addition, we emphasize the subtle ambiguity concerning the observability functional, and we provide an extended statement regarding $Q^A$ as well as the homogeneous case. {Furthermore, we note that a formulation as for $Q^S$ in (b) of \Cref{thm:nullPQ} is not possible for $Q^A$. This can be seen from the proof of (d) of \Cref{thm:nullPQ} where in general $N_k\ker(Q^A)\not\subseteq \ker(Q^A)$. If we now use the homogeneous system \cref{eq:BQO_hom} with a nonzero input $u$, we only obtain 
    \begin{align*}        
        Q^A\dot{x}(t)=Q^AAx(t)+\sum_{k=1}Q^AN_kx(t)u_k(t)= \sum_{k=1}Q^AN_kx(t)u_k(t)\neq 0
    \end{align*}
    for $x(t)\in\ker(Q^A)$ using $A\ker(Q^A)\subseteq \ker(Q^A)$.
    }
    
   The result of this theorem can be interpreted as follows. The observability Gramian $Q^A$ does not yield a statement regarding the observability functional in \cref{eq:obs_funct} when $\alpha >0$. So, it is not as versatile as $Q^S$ which is also applicable to the observability functional in \cref{eq:obs_funct} when $\alpha >0$. Hence, $Q^A$ does not contain as rich of an information as $Q^S$ for the balancing steps, which can also be observed numerically in \Cref{sec:8}.
   For similar analysis for linear systems with quadratic outputs and  for bilinear systems with linear outputs, we refer the reader to, respectively, \cite{benner_lqo_2022} and \cite{benner_lyapunov_2011}. 
    
    Using the result of \Cref{thm:nullPQ}, we observe that the states that lie in $\ker(P)$ or $\ker(Q)$ have no significant influence on the system dynamics.  This insight allows us to transfer the balancing approach from linear systems to BQO systems.     
In particular, theoretically we diagonalize the Gramians in a similar way and then truncate them to eliminate the less important states.

However, in practice, one does not perform full balancing (simultaneous diagonalization) followed by truncation since this requires performing ill-conditioned operations. Instead, one constructs the reduced BT model directly without full balancing. We follow the same approach here.  To begin, we compute the so-called square-root factors of the Gramians as $P=UU^T$ and $Q=LL^T$, using, for example, Cholesky decompositions or suitable low-rank approximations. In practice, one directly computes $U$ and $L$ without ever forming $P$ and $Q$. Based on these square-root factor, one performs balanced truncation as sketched in \Cref{alg5}.
    \begin{algorithm}
    \caption{\texttt{BT\_BQO}$(P,Q)$, Balanced Truncation for BQO systems with Gramians $P$ and $Q$}\label{alg5}
    \begin{algorithmic}[1]
    \Require $A, B, C, N_k, M_j$ as in \eqref{eq:BQO}, $P=UU^T\in\R^{n\times n}$, $Q=LL^T\in\R^{n\times n}$ and a reduced order $r\leq n$
    \Ensure $\hat{A}, \hat{B}, \hat{C}, \hat{N}_k, \hat{M}_j$ as in \eqref{eq:BQOr}
    \State Compute the SVD 
    \begin{align*}
        U^TL=Z\Sigma Y^T=\begin{bmatrix}
            Z_1&Z_2
        \end{bmatrix}\begin{bmatrix}
            \Sigma_1&\\& \Sigma_2
        \end{bmatrix}\begin{bmatrix}
            Y_1^T\\Y_2^T
        \end{bmatrix}
    \end{align*}
    with $Z_1,Y_1\in\R^{n\times r}$ and $\Sigma_1\in\R^{r\times r}$.
    \State Construct $W^T=\Sigma_1^{-1/2}Y_1^TL^T$ and $V=UZ_1\Sigma_1^{-1/2}$.
    \State Compute $\hat{A}=W^TAV, \hat{B}=W^TB, \hat{C}=CV, \hat{N}_k=W^TN_kV$ and $\hat{M}_j=V^TM_jV$. 
    \end{algorithmic}
    \end{algorithm} 

    \begin{remark} \label{rem:hsv}
        \Cref{alg5} uses the SVD of $U^TL$ with the square-root factors of the Gramians, $P=UU^T$ and $Q=LL^T$. We now compare $Q^S$ in \cref{eq:sumqi} with $Q^A$ in \cref{eq:sumqi_D} and their Lyapunov equations \cref{eq:thm1b} and \cref{eq:thm3} to observe that $Q^S\geq Q^A$ in the sense that $Q^S-Q^A\geq 0$, i.e., $Q^S-Q^A$ is symmetric positive semidefinite. Hence, starting from factorizations $Q^S=L_SL_S^T$ and $Q^A=L_AL_A^T$, we have $Q^S=Q^A+\tilde{L}\tilde{L}^T=\begin{bmatrix}L_A \tilde{L}\end{bmatrix}  \begin{bmatrix} L^T_A \\ \tilde{L}^T\end{bmatrix}$. Thus, we obtain 
        \begin{align*}
            \sigma_i(U^TL_A)\leq \sigma_i\left(\begin{bmatrix}U^TL_A & U^T\tilde{L}\end{bmatrix} \right) = \sigma_i(U^TL_S),
        \end{align*}
        for all the Hankel singular values (HSVs) $\sigma_i(U^TL_A)$ (the diagonal entries of $\Sigma$ in \Cref{alg5}). Hence, we obtain smaller Hankel singular values if we use the observability Gramian $Q^A$ instead of $Q^S$. An analogous result can be shown for the truncated Gramians, where $Q^S_T\geq Q^A_T$ and $Q^P_T\geq Q^A_T$.
    \end{remark}

\section{Numerical Experiments}\label{sec:8}
In this section, we illustrate the balanced truncation algorithm, \cref{alg5}, for BQO systems using several examples. We compare different combinations of reachability and observability Gramians, $P$, $Q^S$, $Q^P$, {$Q^M$,} and $Q^A$, respectively. First, we analyze the algorithm using the following combinations of full Gramians: 
\begin{itemize}
\item \texttt{BT\_BQO}$(P,Q^S)$: Using $P$ and $Q^S$, representing the standard approach.
\item \texttt{BT\_BQO}$(P,Q^A)$: Using $P$ and $Q^A$, as an alternative formulation.
{\item \texttt{BT\_BQO}$(P,Q^M)$: Using $P$ and $Q^M$, as a mixed formulation.}
\end{itemize}
Moreover,  we examine the use of truncated Gramians. We test:
\begin{itemize}
\item \texttt{BT\_BQO}$(P_T,Q_T^S)$: Using $P_T$ and $Q_T^S$, representing the standard approach;
\item \texttt{BT\_BQO}$(P_T,Q_T^P)$: Using $P_T$ and $Q_T^P$, following the method proposed in \cite{padhi_2024};
\item \texttt{BT\_BQO}$(P_T,Q_T^A)$: Using $P_T$ and $Q_T^A$, as an alternative formulation.
\end{itemize}

As pointed out in \cite{benner_lyapunov_2011,condon_nonlinear_2005},
rescaling the input variable $u$ can help to ensure the existence of the Gramians. 
In \cite{benner_lyapunov_2011} it is suggested to replace the bilinear state equation \cref{eq:BQO_a1} by one with scaled input $u\to\frac{1}{\gamma}u$, $0<\gamma<1$, which leads to the modified state-space realization:
\begin{subequations}
    \begin{align*}
    \dot{x}(t) &= Ax(t) + \sum_{k=1}^m (\gamma N_k) x(t) \left(\frac{1}{\gamma}u_k(t)\right)+ (\gamma B)\left(\frac{1}{\gamma}u(t)\right)\\
    &=Ax(t) +\sum_{k=1}^m \tilde{N}_k x(t) \tilde{u}_k(t)+ \tilde{B}\tilde{u}(t).
\end{align*}
\end{subequations}
Thus, the Lyapunov equations \cref{eq:thm1a,eq:thm1b,eq:thm3} need to be adapted as

        \begin{align*}  
        A\tilde{P}+\tilde{P}A^T+\gamma^2\sum_{k=1}^mN_k\tilde{P}N_k^T +\gamma^2BB^T =0,   \\         
        A^T\tilde{Q}^S+\tilde{Q}^SA+\gamma^2\sum_{k=1}^mN_k^T\tilde{Q}^SN_k+\sum_{j=1}^pM_j\tilde{P}M_j + C^TC=0,\\
        A^T\tilde{Q}^A+\tilde{Q}^A+\sum_{j=1}^pM_j\tilde{P}M_j + C^TC=0. 
   \end{align*}
By choosing a suitable $\gamma,$ the constant $\Gamma_P$ in Theorem \ref{thm:PQexist} will be smaller and the conditions for the new Gramians to exist are easier to hold.
We see that the Lyapunov equation for $\tilde{Q}^S$ above includes a scaling parameter $\gamma^2$ in front of the term $\sum_{k=1}^mN_k^T\tilde{Q}^SN_k$. This resembles the Lyapunov equation \cref{eq:thm3a}, which describes the Lyapunov equation for the mixed observability Gramian $Q^M$ with $\phi_k=\gamma$. Hence, we can interpret choosing the parameters $\phi_k$ for the Gramian $Q^M$ as an additional scaling which can be done independently of the usual scaling in $u$ with $\gamma$.

Our MATLAB implementation uses solvers from the MESS library to solve the standard and bilinear Lyapunov equations. Specifically, we apply \texttt{mess\_lyap} with default settings, and \texttt{mess\_lyapunov\_bilinear} with the following options: residual norm tolerance \texttt{opts.blyap.res\_tol}$=1e-8$, relative difference tolerance \linebreak[4] \texttt{opts.blyap.rel\_diff\_tol}$=1e-7$, and maximum iterations \texttt{opts.blyap.maxiter}$=50$ (see \cite{Mess}).
The code for the numerical examples is available on ZENODO ( \url{https://doi.org/10.5281/zenodo.15796902}). Computations were run on an Intel(R) Core(TM) i7-1255U CPU @ 1.70 GHz with 16 GB RAM using MATLAB 2024a.

\subsection{Nonlinear RC-example}
Our first example is a nonlinear RC circuit, discussed in \cite{bai_projection_2006,benner_lyapunov_2011}. This nonlinear system is approximated by a bilinear system via the Carleman bilinearization~\cite{rugh_nonlinear_1981}, which significantly increases the system dimension, so that the task of model order reduction appears. We use the original fully nonlinear dimension $200$, which leads to the BQO dimension of $n  = 200^2+200 = 40200$ after bilinearization. This is a SISO example where we choose the output as 
\[
y=[1,0,\dots,0]x+ \frac{1}{200^2}x^T\begin{bmatrix} I_{200} & 0 \\ 0 & 0    
\end{bmatrix}x.
\]
Hereby, the standard linear output from \cite{bai_projection_2006} is extended by a quadratic term. This quadratic term describes the root mean squared error of the original states (RMSE, see, e.g., \cite{reiter_werner2024}). Following the original source~\cite{bai_projection_2006} of this example, we choose a reduced order of $r=21$. (We test various orders later in \Cref{fig:10a}). To test the performance of the reduced models, we simulate the full model and all the reduced models with the input $u(t)=e^{-t}$ and compare the original output $y(t)$ with the corresponding reduced order outputs $y_r(t)$. 

For the first test, we use a scaling parameter of $\gamma =0.1$. Also, we use two different choices for the parameter $\phi$ in $Q^M$. Recall that $Q^M$ with $\phi =1$ corresponds to $Q^S$, and $\phi =0$  to $Q^A$. We choose $\phi =0.1$ and $\phi = 0.5$. The right plot in the top row of \Cref{fig:9} shows the normalized Hankel singular values (HSVs) (i.e., $\sigma_i(A)/\sigma_1(A)$ for the singular values of a matrix $A$), which (as expected) decrease more rapidly for smaller $\phi$, especially at higher indices. 
{This is consistent with \Cref{rem:hsv} because the largest HSVs are nearly identical, and thus, the decay of the normalized HSVs follows that of the original HSVs.}
The left plot in the top row in \Cref{fig:9} shows the relative output errors over time for $t\in[0,2]$. The methods \texttt{BT\_BQO}$(P,Q^S)$ and \texttt{BT\_BQO}$(P,Q^M)$  perform well, while \texttt{BT\_BQO}$(P,Q^A)$ shows noticeably worse results, suggesting it fails to capture the nonlinear dynamics effectively. Similar trends are seen in the bottom row of plots in \Cref{fig:9} for $\gamma =0.5$. But in this case, the HSVs decay more slowly, and the output error grows a little bit more over time. Notably, \Cref{tab:1} illustrates significant differences in overall computation times for constructing the reduced models: smaller $\gamma$ values lead to better-conditioned generalized Lyapunov equations (see \Cref{thm:PQexist}) and faster convergence. In contrast, $\phi$ mainly affects the $Q$-Gramian equation and has a smaller impact on runtime, even though smaller values of $\phi$ still lead to faster balanced truncation overall. For the special case of $\phi=0$, which corresponds to $Q^A$, we only need to solve a standard Lyapunov equation for the $Q$-Gramian. Thus, the computation time is reduced even further.

\begin{figure}[h]
\centering
    \begin{subfigure}[b]{1\textwidth}
    \centering
    \includegraphics[width=1\textwidth]{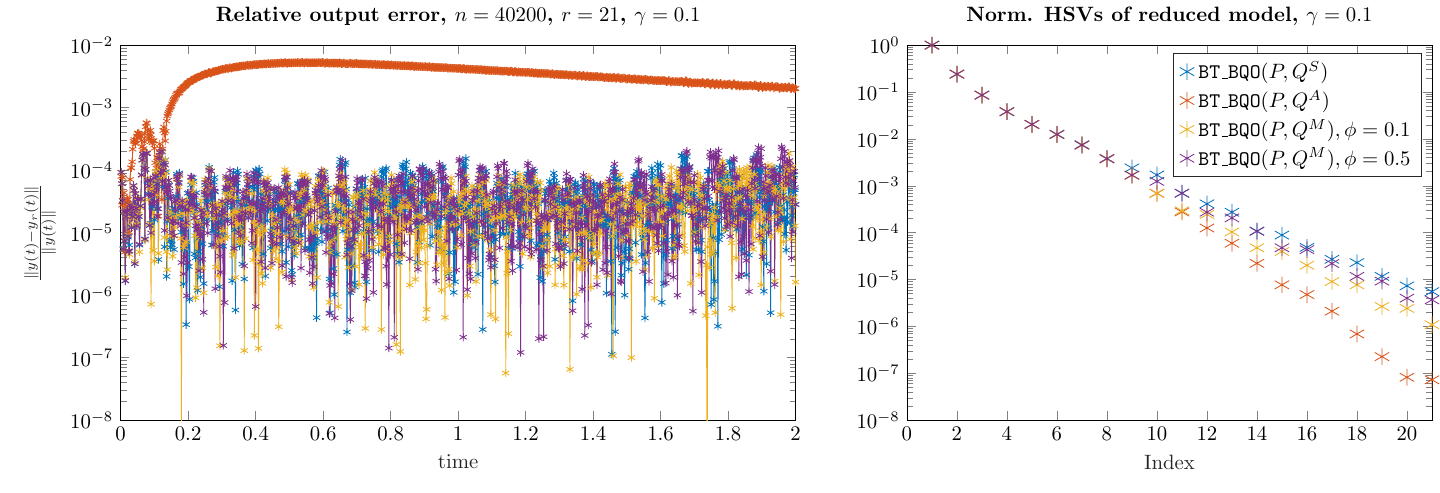}
    \end{subfigure}
    \begin{subfigure}[b]{1\textwidth}
    \centering
    \includegraphics[width=1\textwidth]{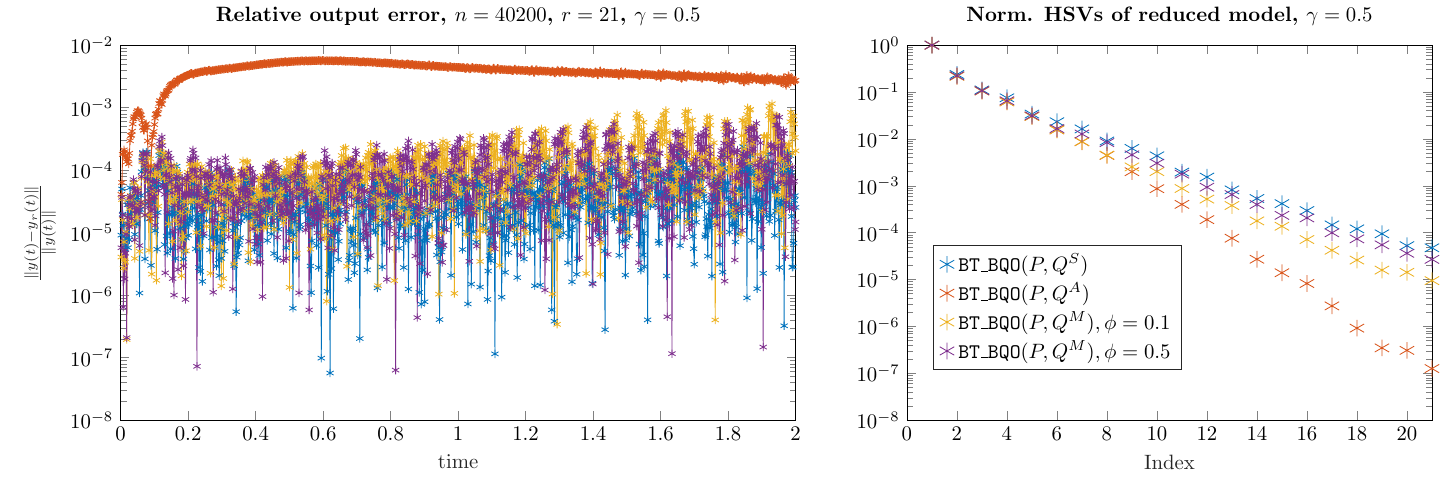}
    \end{subfigure}
    \caption{RC example, full Gramians using $\gamma=0.1$ (first row) and $\gamma=0.5$ (second row)}
    \label{fig:9}
\end{figure}

Next we present the results using the truncated Gramians, with $\gamma = 0.1$ and omitting the $Q^M$-Gramians. {As shown in \Cref{fig:10}, \texttt{BT\_BQO}$(P_T,Q_T^P)$ and \texttt{BT\_BQO}$(P_T,Q_T^S)$ yield almost the same approximation quality as \texttt{BT\_BQO}$(P,Q^S)$ in terms of the relative output error as they largely overlap. This also holds for \texttt{BT\_BQO}$(P_T,Q_T^A)$, which overlaps with \texttt{BT\_BQO}$(P,Q^A)$. We also observe that also the HSVs of the truncated variants match the HSVs of the methods using the truncated variants. Thus, various BT formulations with truncated Gramians performs closely to their regular BT versions with the original Gramians. 
As expected, \Cref{tab:1} shows that using truncated Gramians reduces computation time without sacrificing accuracy, as confirmed by the results in \Cref{fig:10}.

We revisit the output error, but now with an increasing order $r$. We use $N_t=10^3$ equidistant points $t_i,i=1,\dots,N_t$ for the time simulation over $t \in [0,2]$ and denote the discrete-time output of the original system at these points by $Y_i\approx y(t_i)\in\R^2$. Next, we consider the output error in the Frobenius norm as $\frac{\norm{Y-Y^{(r)}}_F}{\norm{Y}_F}$ with $Y=[Y_1, \dots, Y_{N_t}]\in\R^{2\times N_t}$, and similarly, $Y^{(r)}$ denotes the reduced system output. \Cref{fig:10a} shows the evolution of this output error with increasing orders of $r$. We observe again that the methods using the truncated Gramians approximate the methods using the full Gramians accurately. All the BT variants provide high-fidelity approximation except for those using $Q^A$ and its truncated version.

\begin{figure}[h]
    \centering
    \includegraphics[width=\textwidth]{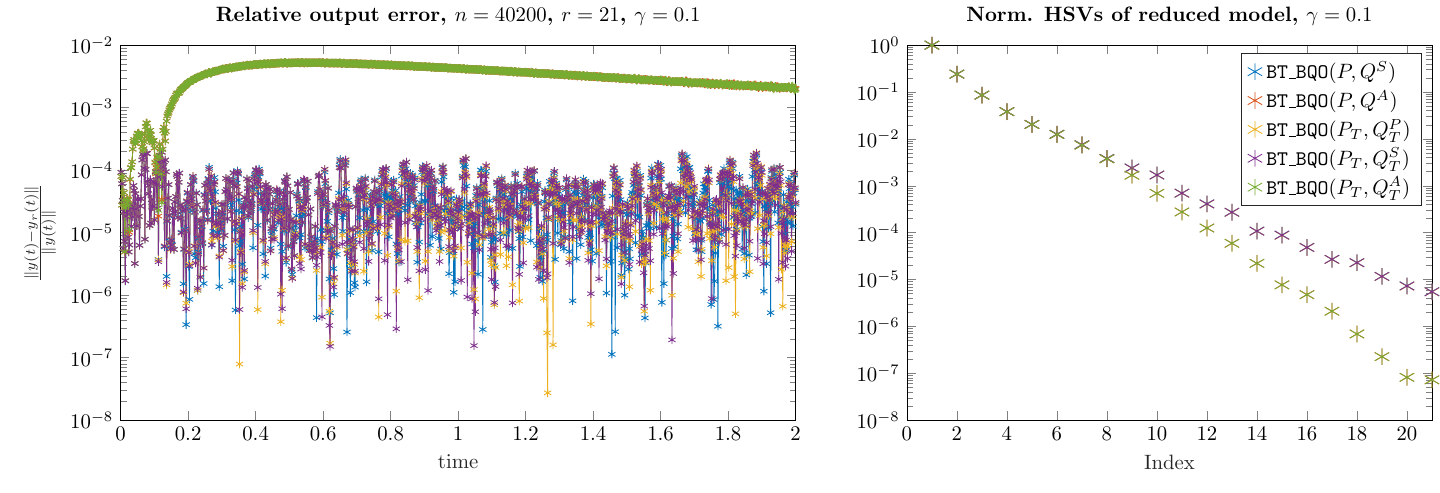}
    \caption{RC example, full vs truncated Gramians}
    \label{fig:10}
\end{figure}

\begin{figure}[h]
    \centering
    \includegraphics[width=0.8\textwidth]{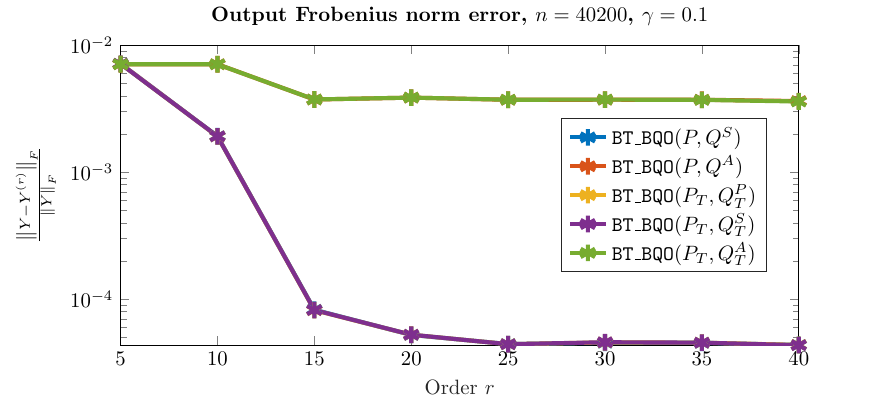}
    \caption{RC example, full vs truncated Gramians}
    \label{fig:10a}
\end{figure}

\begin{table}[htbp]
\footnotesize
\caption{RC example, Timings in sec for $r=21$}\label{tab:1}
\begin{center}
  \begin{tabular}{|c|c|c|c|c|c|c|c|} \hline
  $\gamma$  & \bf $P,Q^S$ & \bf $P,Q^A$ & \bf $P_T,Q_T^P$ & \bf $P_T,Q_T^S$ & \bf $P_T,Q_T^A$ & \bf $P,Q^M, \phi=0.1$ & \bf $P,Q^M, \phi=0.5$\\ \hline
  0.1   &139.6&70.2& 16.3&20.6 &12.2 &108.1 & 125.3 \\ \hline
    0.5   &569.0&254.0& 18.9 & 23.4&13.7 & 512.7& 714.0 \\
     \hline
  \end{tabular}
\end{center}
\end{table}

\subsection{Heat equation}
Our second example is based on a MIMO-example from \cite{benner_lyapunov_2011} describing a bilinear controlled heat transfer system with mixed Dirichlet and Robin boundary conditions:
\begin{align*}
    x_t= &\Delta x \text{ on } \Omega :=[0,1]^2,\\
    n\cdot\nabla x = &u_1(x-1) \text{ on } \Gamma_1:=\{0\}\times [0,1),\\
    n\cdot\nabla x = &u_2(x-1) \text{ on } \Gamma_2:=(0,1]\times\{0\},\\
    x= &0 \text{ on } \Gamma_3:=\{1\}\times (0,1] \text{ and } \Gamma_4:=[0,1)\times \{1\}.
\end{align*}
A finite difference discretization on a equidistant $k\times k$ grid leads to bilinear dynamics. We use $k=50$ and a zero initial condition.

For the output, we study first a system with 
\begin{align*}
    y_1(t)&=Cx(t),\\
    y_2(t)&=x(t)^TMx(t),
\end{align*}
where $C=\frac{1}{k^2}[1,\dots,1]$, $M=\frac{1}{k^4} \mathds{1}_{k^2}$ and $\mathds{1}$ as the matrix with all ones. The first component of the output mimics the average temperature as suggested in \cite{benner_lyapunov_2011} and the second component displays the outer product of all temperatures, {i.e., $\norm{y_1(t)}^2=y_2(t)$}. We perform balanced truncation to a reduced dimension $r=20$. We test the performance of the reduced models choosing $u_j(t)=\cos(j \pi t), j=1,2$ as for inputs and simulate the full and reduced dynamics over the time interval $t\in [0,5]$.

In \Cref{fig:11}, we observe again that the formulations using the Gramian $Q^A$ lack in accuracy compared to the other variants and possess smaller Hankel singular values. The approximated BT variants using the truncated Gramians lie on top of the full variants. With respect to the durations of the methods in \Cref{tab:2}, the variants with the truncated Gramians show again their advantage in needing much less computational effort. Moreover, as we stated earlier, the usage of $Q^A$ allows us to compute only a standard Lyapunov equation for the $Q$-Gramian, thus the computation time reduces.

\begin{figure}[h]
    \centering
    \includegraphics[width=\textwidth]{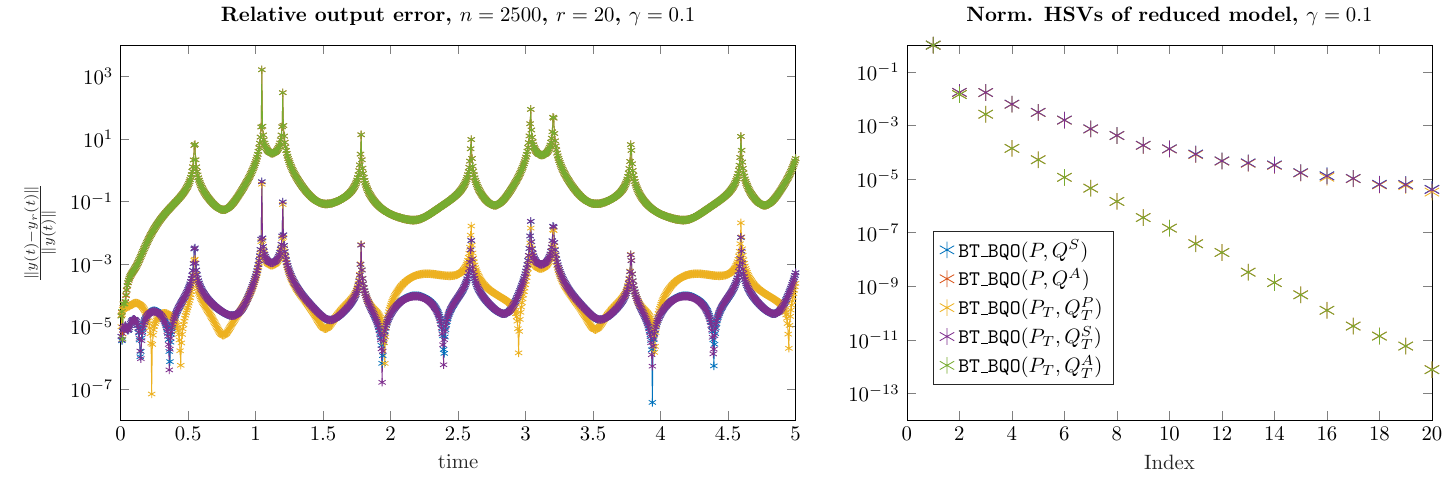}
    \caption{Heat equation example, full vs truncated Gramians, $M=\frac{1}{k^4}\mathds{1}_{k^2}$}
    \label{fig:11}
\end{figure}

Next, we use the average square temperature as the second output
$y_2(t)=x(t)^TMx(t)$ 
where $M=\frac{1}{k^2} I_{k^2}$. This represents the root mean squared error as before in the RC example. Furthermore, we note that the new quadratic output is larger than the one with $M=\frac{1}{k^4}\mathds{1}_{k^2}$, i.e.,  $x(t)^T\left(\frac{1}{k^2}I_{k^2}\right)x(t)\geq x(t)^T\left(\frac{1}{k^4}\mathds{1}_{k^2}\right)x(t)$.

\begin{figure}[h]
    \centering
    \includegraphics[width=\textwidth]{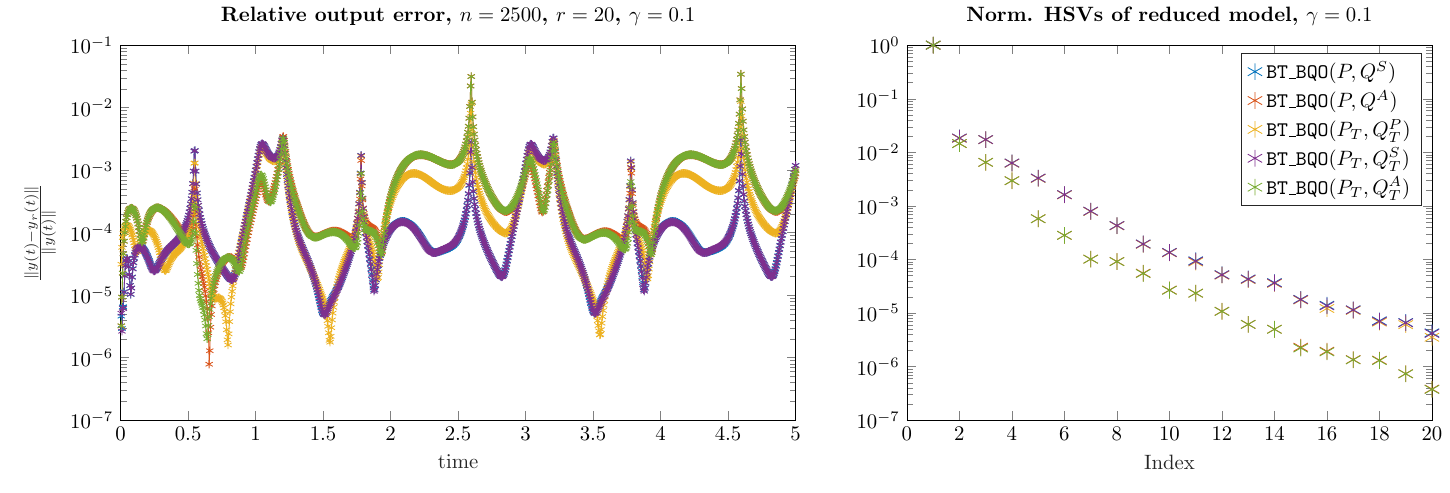}
    \caption{Heat equation example, full vs truncated Gramians, $M=\frac{1}{k^2}I_{k^2}$}
    \label{fig:12}
\end{figure}

Now, we see in the relative output error plot in \Cref{fig:12} that this new output, 
where there are less coupled states due to the diagonal structure of $M$,
enables us to alternatively use the dual formulation $Q^A$ as well. Nevertheless, it is still less accurate than the standard formulation. Regarding the HSVs, we also observe similar results as before, meaning the HSVs using the $Q^A$-Gramian provide smaller HSVs.

\begin{table}[htbp]
\footnotesize
\caption{Heat equation example, Timings in sec for $r=20$, $\gamma=0.1$}\label{tab:2}
\begin{center}
  \begin{tabular}{|c|c|c|c|c|c|} \hline
    & \bf $P,Q^S$ & \bf $P,Q^A$ & \bf $P_T,Q_T^P$ & \bf $P_T,Q_T^S$ & \bf $P_T,Q_T^A$ \\ \hline    
     $M=\frac{1}{k^4}\mathds{1}$ &15.1 & 9.2 & 1.7 & 2.4 & 1.6 \\
     $M=\frac{1}{k^2}I_{k^2}$ &15.8 &  8.7 & 1.9 & 2.1 & 1.0 \\
     \hline
  \end{tabular}
\end{center}
\end{table}

\Cref{fig:13} depicts the output error in the Frobenius norm with increasing orders of reduced dimension $r$ as we did for the heat example. We used again $N_t=10^3$ equidistant points for the time simulation over $t \in [0,2]$. As before, the methods using truncated Gramians yield a good approximation on the methods using the full Gramians. Also, the methods $\texttt{BT\_BQO}(P,Q^A)$ and $\texttt{BT\_BQO}(P_T,Q_T^A)$ show slightly worse errors than the other methods. Furthermore, we see a difference for higher orders where $\texttt{BT\_BQO}(P_T,Q_T^P)$ lacks slightly in accuracy compared to $\texttt{BT\_BQO}(P,Q^S)$ and $\texttt{BT\_BQO}(P_T,Q_T^S)$. This can be explained by the fact that in $\texttt{BT\_BQO}(P_T,Q_T^S)$, the additional Lyapunov solve provides better knowledge of the dynamics (see \Cref{sec:6}).

\begin{figure}[h]
    \centering
    \includegraphics[width=0.8\textwidth]{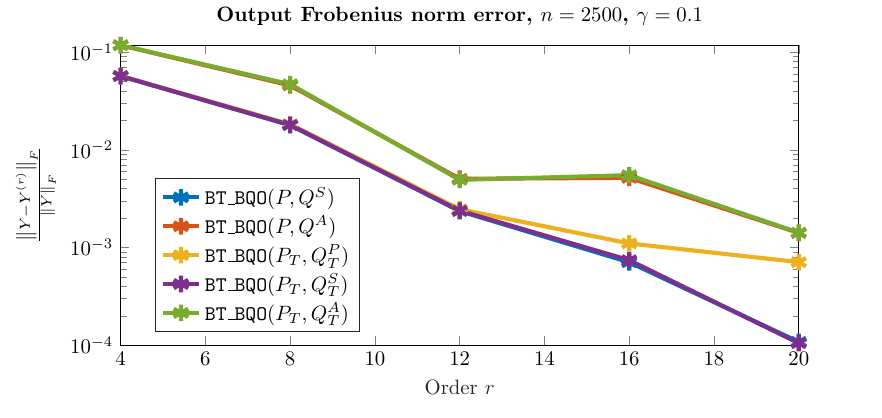}
    \caption{Heat equation example, full vs truncated Gramians, $M=\frac{1}{k^2}I_{k^2}$}
    \label{fig:13}
\end{figure}

 In short, we have seen in these numerical examples that the variants with the formulation of $Q^A$ or $Q^A_T$ work less accurately in most cases, but offer a quicker alternative 
 when the quadratic output has less coupled states.
 The use of truncated Gramians comes with a speedup and does not degrade the accuracy. {For higher reduced orders,  using $\texttt{BT\_BQO}(P_T,Q_T^P)$ may not be as accurate, which only requires two Lyapunov solves for the $Q$-Gramian}. Instead, one may consider $\texttt{BT\_BQO}(P_T,Q_T^S)$, which comes with an additional cost, but provides slightly better approximations (see \Cref{fig:13}).

\section{Conclusion}

In this paper, we studied BQO systems, generalizations of both bilinear and LQO systems. We derived and compared different reachability and observability Gramians by a new interpretation of the formulation for the dual system, which leads to alternative observability Gramians. These new observability Gramians can also be applied to bilinear systems in a similar way. In comparison to \cite{padhi_2024}, we derived the observability Gramian $Q^S$ using a different approach to that used for $Q^P$ in \cite{padhi_2024}. However, if they exist, they both fulfill the same generalized Lyapunov equations. All defined Gramians satisfy similar generalized Lyapunov equations when they exist.  But solving these generalized Lyapunov equations is often  challenging.
 
To address this issue, we proposed truncated variants that reduce the computational burden by requiring only a few standard Lyapunov equations to be solved.
We also introduced a balancing algorithm, analogous to that used for linear systems, and tested it with various Gramian combinations on two numerical examples, one SISO and one MIMO. The results show that the algorithm performs well in both cases. The truncated Gramians approximate the full ones effectively and offer improvements in computational efficiency.
Among the different observability Gramians, balanced truncation using the standard Gramian $Q^S$ outperforms the use of $Q^A$. For higher-order systems, the truncated Gramian $Q_T^P$ may not fully capture the system dynamics, making $Q_T^S$ a better alternative.

\section{Acknowledgements}
The third author warmly thanks Reetish Padhi for the inspiring and insightful discussions about his master's thesis. 

\section{Credit author statement}
 \textbf{Heike Faßbender: }Writing - Original Draft, Conceptualization, Supervision; \textbf{Serkan Gugercin: } Writing - Original Draft, Conceptualization, Methodology; \textbf{Till Peters:} Conceptualization, Methodology, Software, Investigation, Writing - Original Draft

\bibliographystyle{siamplain}
\bibliography{bqo_bt}

\begin{thebibliography}{10}

\bibitem{adrianova_1995}
{\sc L.~Y. Adrianova}, {\em Introduction to linear systems of differential
  equations. {Transl}. from the {Russian} by {Peter} {Zhevandrov}}, vol.~146 of
  Transl. Math. Monogr., Providence, RI: American Mathematical Society, 1995.

\bibitem{al-baiyat_new_1993}
{\sc S.~Al-Baiyat and M.~Bettayeb}, {\em A new model reduction scheme for
  k-power bilinear systems}, in Proceedings of 32nd {IEEE} {Conference} on
  {Decision} and {Control}, Dec. 1993, pp.~22--27 vol.1,
  \url{https://doi.org/10.1109/CDC.1993.325196}.

\bibitem{antoulas_book_2005}
{\sc A.~C. Antoulas}, {\em Approximation of Large-Scale Dynamical Systems},
  Society for Industrial and Applied Mathematics, 2005,
  \url{https://doi.org/10.1137/1.9780898718713}.

\bibitem{bai_projection_2006}
{\sc Z.~Bai and D.~Skoogh}, {\em A projection method for model reduction of
  bilinear dynamical systems}, Linear Algebra and its Applications, 415 (2006),
  pp.~406--425, \url{https://doi.org/10.1016/j.laa.2005.04.032}.

\bibitem{BenB13}
{\sc P.~Benner and T.~Breiten}, {\em Low rank methods for a class of
  generalized {Lyapunov} equations and related issues}, Numer. Math., 124
  (2013), pp.~441--470, \url{https://doi.org/10.1007/s00211-013-0521-0}.

\bibitem{BenB17}
{\sc P.~Benner and T.~Breiten}, {\em Chapter 6: Model Order Reduction Based on
  System Balancing}, SIAM, 2017, pp.~261--295,
  \url{https://doi.org/10.1137/1.9781611974829.ch6}.

\bibitem{BenD11}
{\sc P.~Benner and T.~Damm}, {\em Lyapunov equations, energy functionals, and
  model order reduction of bilinear and stochastic systems}, SIAM J. Control
  Optim., 49 (2011), pp.~686--711, \url{https://doi.org/10.1137/09075041X}.

\bibitem{benner_lyapunov_2011}
{\sc P.~Benner and T.~Damm}, {\em Lyapunov {Equations}, {Energy} {Functionals},
  and {Model} {Order} {Reduction} of {Bilinear} and {Stochastic} {Systems}},
  SIAM Journal on Control and Optimization, 49 (2011), pp.~686--711,
  \url{https://doi.org/10.1137/09075041X}.
\newblock Publisher: Society for Industrial and Applied Mathematics.

\bibitem{benner_modellreduktion_2024}
{\sc P.~Benner and H.~Faßbender}, {\em Modellreduktion: {Eine}
  systemtheoretisch orientierte {Einführung}}, Springer {Studium} {Mathematik}
  ({Master}), Springer Berlin Heidelberg, Berlin, Heidelberg, 2024,
  \url{https://doi.org/10.1007/978-3-662-67493-2}.
\newblock ISSN: 2509-9310, 2509-9329.

\bibitem{BenG24}
{\sc P.~Benner and P.~Goyal}, {\em Balanced truncation for quadratic-bilinear
  control systems}, Advances in Computational Mathematics, 50 (2024), p.~88,
  \url{https://doi.org/10.1007/s10444-024-10186-9}.

\bibitem{benner_lqo_2022}
{\sc P.~Benner, P.~Goyal, and I.~P. Duff}, {\em Gramians, {Energy}
  {Functionals}, and {Balanced} {Truncation} for {Linear} {Dynamical} {Systems}
  {With} {Quadratic} {Outputs}}, IEEE Transactions on Automatic Control, 67
  (2022), pp.~886--893, \url{https://doi.org/10.1109/TAC.2021.3086319}.

\bibitem{BenGR17}
{\sc P.~Benner, P.~Goyal, and M.~Redmann}, {\em Truncated Gramians for Bilinear
  Systems and Their Advantages in Model Order Reduction}, Springer
  International Publishing, Cham, 2017, pp.~285--300,
  \url{https://doi.org/10.1007/978-3-319-58786-8_18}.

\bibitem{BreS21}
{\sc T.~Breiten and T.~Stykel}, {\em Balancing-related model reduction
  methods}, De Gruyter, Berlin, Boston, 2021, pp.~15--56,
  \url{https://doi.org/doi:10.1515/9783110498967-002}.

\bibitem{condon_nonlinear_2005}
{\sc M.~Condon and R.~Ivanov}, {\em Nonlinear systems – algebraic gramians
  and model reduction}, COMPEL - The international journal for computation and
  mathematics in electrical and electronic engineering, 24 (2005),
  pp.~202--219, \url{https://doi.org/10.1108/03321640510571147}.

\bibitem{damm_rational_2004}
{\sc T.~Damm}, {\em Rational {Matrix} {Equations} in {Stochastic} {Control}},
  Lecture Notes in Control and Information Sciences, 297 (2004).

\bibitem{damm_direct_2008}
{\sc T.~Damm}, {\em Direct methods and {ADI}-preconditioned {Krylov} subspace
  methods for generalized {Lyapunov} equations}, Numerical Linear Algebra with
  Applications, 15 (2008), pp.~853--871, \url{https://doi.org/10.1002/nla.603}.

\bibitem{fujimoto_hamiltonian_2000}
{\sc K.~Fujimoto, J.~M. Scherpen, and W.~Gray}, {\em Hamiltonian realizations
  of nonlinear adjoint operators}, Automatica, 38 (2002), pp.~1769--1775,
  \url{https://doi.org/https://doi.org/10.1016/S0005-1098(02)00079-1}.

\bibitem{GolubVanLoan4th}
{\sc G.~H. Golub and C.~F. Van~Loan}, {\em Matrix Computations - 4th Edition},
  Johns Hopkins University Press, Philadelphia, PA, 2013,
  \url{https://doi.org/10.1137/1.9781421407944}.

\bibitem{gray_energy_1998}
{\sc W.~S. Gray and J.~Mesko}, {\em Energy {Functions} and {Algebraic}
  {Gramians} for {Bilinear} {Systems}}, IFAC Proceedings Volumes, 31 (1998),
  pp.~101--106, \url{https://doi.org/10.1016/S1474-6670(17)40318-1}.

\bibitem{GugA04}
{\sc S.~Gugercin and A.~Antoulas}, {\em A survey of model reduction by balanced
  truncation and some new results}, International Journal of Control, 77
  (2004), pp.~748--766, \url{https://doi.org/10.1080/00207170410001713448}.

\bibitem{hart13}
{\sc C.~Hartmann, C.~Sch\"afer-Bung, and A.~Th\"ons-Zueva}, {\em Balanced
  averaging of bilinear systems with applications to stochastic control}, SIAM
  J. on Control and Optimization, 51 (2013), pp.~2356--2378.

\bibitem{HorJ94}
{\sc R.~A. Horn and C.~R. Johnson}, {\em Topics in matrix analysis. 1st
  paperback ed. with corrections}, Cambridge: Cambridge University Press, 1st
  pbk with corr.~ed., 1994.

\bibitem{LanT85}
{\sc P.~Lancaster and M.~Tismenetsky}, {\em The theory of matrices. 2nd ed.,
  with applications}, Computer {Science} and {Applied} {Mathematics}. {Orlando}
  etc.: {Academic} {Press} ({Harcourt} {Brace} {Jovanovich}, {Publishers}).
  {XV} (1985)., 1985.

\bibitem{MehU23}
{\sc V.~Mehrmann and B.~Unger}, {\em Control of port-{H}amiltonian
  differential-algebraic systems and applications}, Acta Numerica, 32 (2023),
  pp.~395--515, \url{https://doi.org/10.1017/S0962492922000083}.

\bibitem{mohler}
{\sc R.~Mohler}, {\em {Nonlinear systems (vol. 2): applications to bilinear
  control}}, Prentice-Hall, Inc. Upper Saddle River, NJ, USA, 1991.

\bibitem{moore_bt}
{\sc B.~Moore}, {\em Principal component analysis in linear systems:
  Controllability, observability, and model reduction}, IEEE Transactions on
  Automatic Control, 26 (1981), pp.~17--32,
  \url{https://doi.org/10.1109/TAC.1981.1102568}.

\bibitem{mullis_roberts}
{\sc C.~Mullis and R.~Roberts}, {\em Synthesis of minimum roundoff noise fixed
  point digital filters}, IEEE Transactions on Circuits and Systems, 23 (1976),
  pp.~551--562, \url{https://doi.org/10.1109/TCS.1976.1084254}.

\bibitem{padhi_2024}
{\sc R.~Padhi}, {\em Model {Order} {Reduction} of {Nonlinear} {Dynamical}
  {Systems}}, master's thesis, Indian Institute of Science Education and
  Research Pune, Pashan, Pune India, May 2024.

\bibitem{Pul23}
{\sc R.~Pulch}, {\em Energy-based model order reduction for linear stochastic
  {G}alerkin systems of second order}, PAMM, 23 (2023), p.~e202300038,
  \url{https://doi.org/10.1002/pamm.20230003833}.

\bibitem{PulA19}
{\sc R.~Pulch and A.~Narayan}, {\em Balanced truncation for model order
  reduction of linear dynamical systems with quadratic outputs}, SIAM Journal
  on Scientific Computing, 41 (2019), pp.~A2270--A2295,
  \url{https://doi.org/10.1137/17M1148797}.

\bibitem{reiter_werner2024}
{\sc S.~Reiter and S.~W.~R. Werner}, {\em Interpolatory model order reduction
  of large-scale dynamical systems with root mean squared error measures},
  2024, \url{https://arxiv.org/abs/2403.08894}.

\bibitem{ReiW24}
{\sc S.~Reiter and S.~W.~R. Werner}, {\em Interpolatory model reduction of
  dynamical systems with root mean squared error}, IFAC-PapersOnLine, 59
  (2025), pp.~385--390, \url{https://doi.org/10.1016/j.ifacol.2025.03.066}.

\bibitem{rugh_nonlinear_1981}
{\sc W.~J. Rugh}, {\em Nonlinear {System} {Theory}: {The} {Volterra}/{Wiener}
  {Approach}}, Johns Hopkins University Press, 1981.

\bibitem{Mess}
{\sc J.~Saak, M.~K\"{o}hler, and P.~Benner}, {\em {M-M.E.S.S.}-3.0 -- the
  matrix equations sparse solvers library}, Aug. 2023,
  \url{https://doi.org/10.5281/zenodo.7701424}.
\newblock see also:\url{https://www.mpi-magdeburg.mpg.de/projects/mess}.

\bibitem{scherpen_balancing_1993}
{\sc J.~M.~A. Scherpen}, {\em Balancing for nonlinear systems}, Systems \&
  Control Letters, 21 (1993), pp.~143--153,
  \url{https://doi.org/10.1016/0167-6911(93)90117-O}.

\bibitem{ShaSS16}
{\sc S.~D. Shank, V.~Simoncini, and D.~B. Szyld}, {\em Efficient low-rank
  solution of generalized {Lyapunov} equations}, Numer. Math., 134 (2016),
  pp.~327--342, \url{https://doi.org/10.1007/s00211-015-0777-7}.

\bibitem{Son98}
{\sc E.~D. Sontag}, {\em Mathematical control theory. {Deterministic} finite
  dimensional systems.}, vol.~6 of Texts Appl. Math., New York, NY: Springer,
  2nd ed.~ed., 1998.

\bibitem{VanVNLM12}
{\sc R.~Van~Beeumen, K.~Van~Nimmen, G.~Lombaert, and K.~Meerbergen}, {\em Model
  reduction for dynamical systems with quadratic output}, International Journal
  for Numerical Methods in Engineering, 91 (2012), pp.~229--248,
  \url{https://doi.org/10.1002/nme.4255}.

\bibitem{Van06}
{\sc A.~van~der Schaft}, {\em Port-{H}amiltonian systems: an introductory
  survey}, in International congress of mathematicians, European Mathematical
  Society Publishing House (EMS Ph), 2006, pp.~1339--1365,
  \url{https://doi.org/10.4171/022-3/65}.

\bibitem{willems70}
{\sc J.~L. Willems}, {\em Stability Theory of Dynamical Systems}, Wiley, New
  York, 1970.

\bibitem{YueM13}
{\sc Y.~Yue and K.~Meerbergen}, {\em Accelerating optimization of parametric
  linear systems by model order reduction}, SIAM Journal on Optimization, 23
  (2013), pp.~1344--1370, \url{https://doi.org/10.1137/120869171}.

\bibitem{zhang_2002}
{\sc L.~Zhang and J.~Lam}, {\em On \textit{{H}}2 model reduction of bilinear
  systems}, Automatica, 38 (2002), pp.~205--216,
  \url{https://doi.org/10.1016/S0005-1098(01)00204-7}.

\bibitem{ZhoDG96}
{\sc K.~Zhou, J.~C. Doyle, and K.~Glover}, {\em Robust and optimal control},
  Upper Saddle River, NJ: Prentice Hall, 1996.

\end{thebibliography}

\end{document}